\documentclass[10pt,letterpaper,oneside]{amsart}

\usepackage[utf8]{inputenc}
\usepackage{bm}
\usepackage{geometry}
\usepackage{tikz}
\usetikzlibrary{calc}
\usetikzlibrary{arrows.meta, positioning, quotes}
\usepackage{subcaption}
\usepackage{caption}
\usepackage{amsfonts,amsmath,amssymb,amsthm,amsxtra,amsbsy,mathtools,mathrsfs,dsfont,lmodern,microtype,longtable,verbatim}
\usepackage[colorlinks=true, allcolors=blue]{hyperref}

\theoremstyle{plain}
\usepackage{bbm}
\usepackage[most]{tcolorbox}
\usepackage{mdframed}
\usepackage{xcolor}
\usepackage[T1]{fontenc}
\usepackage{hyperref}
\usepackage{enumerate}
\usepackage{graphicx}
\allowdisplaybreaks

\newmdtheoremenv[
  linecolor=blue,
  backgroundcolor=blue!5,
  leftmargin=0,
  rightmargin=0,
  skipabove=10pt,
  skipbelow=10pt,
  innerleftmargin=10pt,
  innerrightmargin=10pt,
  innertopmargin=10pt,
  innerbottommargin=10pt
]{theorem}{Theorem}[section]

\newmdtheoremenv[
  linecolor=green!50!black,
  backgroundcolor=green!5,
  leftmargin=0,
  rightmargin=0,
  skipabove=10pt,
  skipbelow=10pt,
  innerleftmargin=10pt,
  innerrightmargin=10pt,
  innertopmargin=10pt,
  innerbottommargin=10pt
]{definition}{Definition}[section]

\newmdtheoremenv[
  linecolor=gray,
  backgroundcolor=gray!10,
  leftmargin=0,
  rightmargin=0,
  skipabove=10pt,
  skipbelow=10pt,
  innerleftmargin=10pt,
  innerrightmargin=10pt,
  innertopmargin=10pt,
  innerbottommargin=10pt
]{remark}{Remark}[section]

\newmdtheoremenv[
  linecolor=purple,
  backgroundcolor=purple!10,
  leftmargin=0,
  rightmargin=0,
  skipabove=10pt,
  skipbelow=10pt,
  innerleftmargin=10pt,
  innerrightmargin=10pt,
  innertopmargin=10pt,
  innerbottommargin=10pt
]{lemma}{Lemma}[section]

\newmdtheoremenv[
  linecolor=orange!70!black,
  backgroundcolor=orange!5,
  leftmargin=0,
  rightmargin=0,
  skipabove=10pt,
  skipbelow=10pt,
  innerleftmargin=10pt,
  innerrightmargin=10pt,
  innertopmargin=10pt,
  innerbottommargin=10pt
]{proposition}{Proposition}[section]

\newmdtheoremenv[
  linecolor=teal,
  backgroundcolor=teal!5,
  leftmargin=0,
  rightmargin=0,
  skipabove=10pt,
  skipbelow=10pt,
  innerleftmargin=10pt,
  innerrightmargin=10pt,
  innertopmargin=10pt,
  innerbottommargin=10pt
]{corollary}{Corollary}[section]

\numberwithin{equation}{section}

\newcommand{\R}{\mathbb{R}}

\DeclareMathOperator{\Id}{Id}

\title[From CBO to PSO: Convergence Guarantees under Drift-Diffusion Coupling]{From Consensus-Based Optimization\\ to Particle Swarm Optimization: \\
Convergence Guarantees under Drift-Diffusion Coupling}

\author{Franca Hoffmann}
\address{Department of Computing and Mathematical Sciences, California Institute of Technology} 
\email{franca.hoffmann@caltech.edu} 

\author{Dohyeon Kim}
\address{Department of Computing and Mathematical Sciences, California Institute of Technology}
\email{dohyeon@caltech.edu} 

\author{Ritvik Teegavarapu}
\address{Division of Applied Mathematics, Brown University}
\email{rteegava@brown.edu}

\date{\today}

\allowdisplaybreaks

\begin{document}

\begin{abstract}
Particle swarm optimization (PSO) is a widely used algorithm featured in many state-of-the-art optimization tool-kits. However, rigorous performance guarantees are still lacking. The standard PSO dynamics do not admit a natural mean-field description, which would provide an avenue for theoretical analysis. By modifying the PSO formulation, one can recover the consensus-based optimization (CBO) algorithm with memory, which admits a mean-field limit and facilitates rigorous convergence analysis. These theoretical guarantees rely heavily on the fact that for CBO, the drift and noise strengths can be chosen independently, whereas they are coupled for PSO. We analyze how the PSO parameter coupling affects existing convergence guarantees for CBO and its variant with memory effect. We show, by an explicit construction, that the coupling still leaves a non-empty set of admissible parameters for these convergence guarantees to hold. However, the admissible parameter ranges shrink in the limits used to recover PSO. The resulting convergence guarantees from CBO therefore do not directly extend to the classical PSO model. We provide numerical simulations illustrating the parameter tradeoffs shown in the theoretical analysis. 
\end{abstract}

\maketitle

\section{Introduction}\label{sec1}
Many physical and biological systems consist of large collections of interacting agents, such as neurons in the brain \cite{BaladronFasoliFaugerasTouboul2012}, birds in a flock
\cite{cucker2007, meancuck2009, VicsekEtAl1995, BalleriniEtAl2008}, or atoms in
a lattice \cite{Spohn1991}. Formally, these \emph{interacting particle systems} describe the evolution of many individual agents through stochastic differential equations whose interactions drive emergent macroscopic behavior. Some computational approaches known as \emph{metaheuristic methods} \cite{DreoEtAl2006, zito2025Metaheuristics} leverage multiple interacting agents to solve global optimization problems. These methods balance exploration of the search space and exploitation of already detected promising regions by allowing particles to exchange information about their positions and performance. The general objective is to find the global minimizer of a given cost function $J:\R^d\to\R$,
$$x^* = \arg\min_{x \in \mathbb{R}^d} J(x).$$
We consider $N$ particles evolving on the landscape of $J(x)$ according to coupled stochastic differential equations (SDEs). Each particle’s motion combines random exploration with attraction toward favorable regions based on information shared within the ensemble. In particular, we wish to design algorithms that do not need access to the gradient of the cost function as its derivatives may not be available, or too expensive to compute. This is of particular interest in real-world problems where classical algorithms like gradient descent fail to converge reasonably fast when operating on high-dimensional systems, or in the setting of Bayesian inverse problems, where the target measure is often only accessible as a zero-order oracle \cite{Stuart2010Bayesian, IglesiasLawStuart2013}. Among interacting particle algorithms, \emph{particle swarm optimization} (PSO) \cite{kennedy95particle} and \emph{consensus-based optimization} (CBO) \cite{PinnauTotzeckTseMartin2017} are two related families of algorithms. PSO is a classical and widely implemented method for non-convex optimization, used extensively in engineering applications, and it comes with a substantial stability and parameter-selection literature of its own: \cite{ClercKennedy2002} derives the constriction coefficients that keep the swarm from diverging, and \cite{Trelea2003} analyzes convergence of the deterministic recursion and the resulting parameter regions, with related trajectory and stochastic-stability analyses in \cite{VanDenBerghEngelbrecht2006, JiangLuoYang2007}. That literature constrains the same coefficients $c_1, c_2$ we study here, but through the stability of the discrete recursion rather than through an approximate mean-field decay estimate. CBO, on the other hand, is a more recent formulation inspired by simulated annealing, in which the Laplace principle is used to locate the low-objective region of the search space, and hence an approximate minimizer, rather than approximating a gradient. Further, the fact that CBO has a well-defined mean-field description facilitates rigorous analysis of convergence properties \cite{HuangQiu2022MeanFieldCBO}. 
\newline
\newline
The situation for PSO is more delicate, and we distinguish two formulations throughout. The classical algorithm \eqref{eqn:pso_all}, in which the global and personal bests are selected by an $\mathsf{argmin}$ over the ensemble, is not directly amenable to standard mean-field analysis, since those selections are discontinuous in the particle positions. This is the version most commonly implemented in available optimization toolboxes, and for which very little to no theoretical guarantees exist still to this date. Regularized continuous-time formulations, in which the global best is replaced by a weighted mean and the personal best by a smooth surrogate, do admit mean-field descriptions \cite{grassi2020particle, Huang2021MeanFieldPSO, grassi2021meanfieldparticleswarmoptimization} and do come with convergence results \cite{Huang_2023}; it is such a formulation that we analyze here. The classical discontinuous (with respect to particle positions due to the $\mathsf{argmin}$) algorithm itself has recently been taken up directly in \cite{BorghiHuangKim2026}, and how much is given up in passing to the regularized model is the subject of Sections~\ref{sec4}--\ref{sec5} below.
\newline
\newline
In PSO, each particle updates its velocity and position using two sources of information: the \emph{global best} position across all particles at the present time, and the \emph{personal best} position along its own trajectory up to the present time. The dynamics are inherently second-order, with inertia, a drift toward these best positions, and stochastic perturbations. However, this classical formulation is not directly amenable to standard mean-field analysis, because the minimum operation used to select the best positions is discontinuous in the particle positions; the reductions listed next are what restore the continuity such an analysis needs. CBO can be recovered from PSO through a series of reductions (see Section \ref{sec2} for more details):
\begin{enumerate}
    \item Remove the attraction toward each particle's personal-best position and thus, remove the memory effect by suppressing the personal-best term.
    \item Replace the discrete global-best position with a smooth, weighted average (a ``softmin'' approximation) over the ensemble.
    \item Take the small-inertia limit to obtain a first-order system. 
\end{enumerate}
Each of these three modifications has been studied independently in the literature. Reduction (1) is the most drastic, since it discards the memory that distinguishes PSO from CBO in the first place, but it is a defensible approximation: numerical studies of the swarm report that the attraction toward the global best typically dominates the attraction toward the personal best over most of a run, suggesting that removing the personal-best term has a limited effect on the qualitative behavior of the dynamics \cite{grassi2020particle, grassi2021meanfieldparticleswarmoptimization}. This is also what makes the memory-augmented models \cite{TotzeckWolfram2020, BorghiGrassiPareschi2023, Riedl2024Memory} a refinement rather than a correction of CBO. Reduction (2) is the softmin approximation justified by the Laplace principle, and Reduction (3) has been established rigorously as a zero-inertia limit in \cite{CiprianiHuangQiu2022}. These modifications yield dynamics that are continuous and analytically tractable, allowing the use of PDE analysis and mean-field tools to establish convergence to the global minimizer. One key difference between the PSO and CBO algorithms is that the latter allows $\textit{independent}$ control of the drift and diffusion strengths, while the former assumes them coupled. In particular, this independent control allows for more freedom in the behavior of CBO. 
\newline
\newline
The first direction of existing literature is the classical analysis of the discrete PSO recursion \eqref{eqn:pso_all}, where stability and parameter-selection results \cite{ClercKennedy2002, Trelea2003, VanDenBerghEngelbrecht2006, JiangLuoYang2007} identify coefficient regions in which particle trajectories remain bounded or converge; see \cite{ShiEberhart1998} for the inertia-weight variant these analyses are usually stated for, and \cite{PoliKennedyBlackwell2007, BonyadiMichalewicz2017} for surveys. The second is the passage from PSO to CBO: \cite{grassi2020particle} performs this at the level of stochastic modeling and obtains a mean-field limit for the regularized system, \cite{Huang2021MeanFieldPSO, grassi2021meanfieldparticleswarmoptimization} develop the mean-field formulation further, and \cite{CiprianiHuangQiu2022} makes the small-inertia step rigorous as a zero-inertia limit. The third is the mean-field theory of CBO itself, from the original model and its formal limit \cite{PinnauTotzeckTseMartin2017} and the analytical framework of \cite{carrillo2018analytical}, through the anisotropic variant \cite{CarrilloJinLiZhu2021, FornasierKlockRiedl2022}, the global-convergence results of \cite{Fornasier_2024}, and the first-order and time-discrete analyses of \cite{HaJinKim2020, HaJinKim2021, HaHwangKim2024}, to the propagation-of-chaos results that justify the mean-field description \cite{HuangQiu2022MeanFieldCBO, Gerber_2025, gerber2026uniformintimepropagationchaosconsensusbased}; \cite{Totzeck2022Trends} surveys the area. The fourth is the memory-augmented variants that reinstate the personal best \cite{TotzeckWolfram2020, BorghiGrassiPareschi2023, Riedl2024Memory}, of which \cite{Huang_2023} supplies the convergence theorem we work with.  What the present paper adds is orthogonal to all four: rather than proving a new convergence result, we ask how \emph{parameter couplings in PSO} affect the \emph{convergence guarantees} known for CBO and its \emph{memory-augmented variant}. It is not necessarily true that existing results still hold when the strengths of the drift and the diffusion cannot be chosen independently, and we will qualitatively investigate the convergence results of \cite{PinnauTotzeckTseMartin2017, Huang_2023} under this loss of degrees of freedom. We highlight the difficulty this poses in showing the existence of a parameter set under which convergence guarantees hold, and discuss potential ways to overcome them with a Wasserstein control approach as done in \cite{Fornasier_2024}. 
\newline 
\newline 
In Section \ref{sec2}, we present the various models (PSO, CBO, CBO with memory effect) to illustrate the reductions needed to relate the models. In Section \ref{sec3}, we show that there is still an admissible set of parameters for a regularized first order version of PSO under the convergence guarantees for the first-order CBO model. In Section \ref{sec4}, we show that there exists an admissible set of parameters for CBO with memory when the same coupling of parameters as in PSO is used. In Section \ref{sec5}, we summarize the most important tradeoffs in the parameter relations. In Section \ref{sec6}, we present numerical simulations illustrating our theoretical results from the previous sections. In Section \ref{sec7}, we discuss future directions to show convergence guarantees for the original PSO model. 

\section{Methods} \label{sec2}
\subsection{Models}
We first write down the original PSO model as presented in \cite{grassi2020particle}. The matrices $R_1^n, R_2^n \in \mathbb{R}^{d\times d}$ are diagonal matrices with uniformly distributed diagonal entries on the interval $[0,1]$, and are generated for each particle and time iteration. The particle positions $\{x_i\}_{i=1}^N$ and their velocities $\{v_i\}_{i=1}^{N}$ evolve according to
\begin{subequations}\label{eqn:pso_all}
\begin{equation}\label{eqn:pso}
\begin{split}
x^{n+1}_i &= x^n_i + v^{n+1}_i, \\ 
v^{n+1}_i &= v^{n}_i + c_1R_1^n(y^n_i - x^n_i) + c_2R_2^n(\overline{y}^n - x^n_i).
\end{split}
\end{equation}
The $N$ particles are each accelerated towards two points, namely the \emph{personal} best $y^n_i$ and the \emph{global best} $\overline{y}^n$ with respective strengths $c_1, c_2 > 0$, defined below. 
\begin{equation}\label{eqn:pso_points}
\begin{split}
y^{0}_i &= x_i^0, \\ 
y^{n+1}_i &= \begin{cases} y_i^n & \text{if} \ J(x_i^{n+1}) \geq J(x_i^n), \\ x_i^{n+1} & \text{if} \ J(x_i^{n+1}) < J(x_i^n), \\ \end{cases} \\ 
\overline{y}^{0} &= \mathsf{argmin}\{J(x_1^0), J(x_2^0), \cdots, J(x_N^0)\}, \\ 
\overline{y}^{n+1} &= \mathsf{argmin}\{J(x_1^{n+1}), J(x_2^{n+1}), \cdots, J(x_N^{n+1}), J(\overline{y}^n)\}. \\ 
\end{split}
\end{equation}
\end{subequations}
We note that by decomposing the random matrices $R_1^n, R_2^n$ into a deterministic component given by the mean $\frac12 \Id$ and a stochastic component given by the fluctuations $(R_i^n-\frac 12 \Id)$ around the mean, the velocity update of \eqref{eqn:pso} splits into a deterministic drift term and a zero-mean stochastic term, and passing to a diffusive limit, one can derive a time-continuous formulation of PSO by introducing the memory effect strength $\kappa > 0$ \cite{grassi2020particle}:
\begin{subequations}\label{eqn:pso_sde_all}
\begin{equation}\label{eqn:pso_sde}
\begin{split}
dX^i_t &= V^i_t \ dt,\\ 
dY^i_t &= \kappa (X_t^i - Y_t^i) \cdot S(X_t^i, Y_t^i) \ dt, \\
\overline{Y}_t &= \mathsf{argmin}\{J(Y_t^1), J(Y_t^2), \dots, J(Y_t^N)\}, \\
dV^i_t &= \frac{c_1}{2} (Y_t^i - X_t^i) \ dt + \frac{c_2}{2} (\overline{Y}_t - X_t^i) \ dt + \frac{c_1}{2\sqrt{3}} D(Y_t^i - X_t^i) \ dB^{1,i}_t + \frac{c_2}{2\sqrt{3}} D(\overline{Y}_t - X_t^i) \ dB^{2,i}_t. \\ 
\end{split}
\end{equation}
The term $D(X_t)$ represents the anisotropic diffusion, $B_t^{k,i}$ is an ensemble of $d$-dimensional Brownian motions for $k = 1,2$ and $i = 1, \dots, N$, and $S(x,y)$ is the shifted Heaviside function of the objective: 
\begin{equation}\label{eqn:time_cont}
\begin{split}
D(X_t) &= \text{diag}\{(X_t)_1, (X_t)_2, \cdots, (X_t)_d\} \in \mathbb{R}^{d \times d}  \qquad S(x,y) = 1 + \text{sign}(J(y) - J(x)).
\end{split}
\end{equation}
\end{subequations}
We emphasize that while PSO retains a memory effect in the evolution due to the presence of the approximate personal best $Y_t^i$, it couples together the drift and diffusion strength by $c_1$ and $c_2$ for each, the personal best $Y_t^i$ and the global best $\overline{Y}_t$.
\newline
\newline
If one wanted to ignore the memory effect entirely and thus remove the personal best from the dynamics (see Step 1 above), it is sufficient to let $c_1 = 0$ in the dynamics written in \eqref{eqn:pso_sde}. Even then, PSO lacks a mean-field description due to the presence of the minimum in the computation of the global best $\overline{Y}_t$, which is non-continuous in the particle positions. To avoid this problem, we use the classical Laplace's principle to approximate the global best (Step 2 above), which tells us that most of the mass of the measure concentrates around the global minimizer of the function as the \emph{inverse temperature} ($\alpha$) is taken large. 

\begin{theorem}[Laplace's Principle \cite{DemboZeitouni2010}]
For any compactly supported $\rho \in \mathcal{P}(\mathbb{R}^d)$,
$$\lim_{\alpha \rightarrow \infty} \left(\frac{-1}{\alpha} \cdot \log\left(\int_{\mathbb{R}^d} \exp(-\alpha J(x)) \ \text{d}\rho\right)\right) = \inf_{x \in \mathsf{supp}(\rho)} J(x)$$
\end{theorem} 
\vspace{0.5em}
Next, let us describe the CBO dynamics as presented in \cite{CarrilloJinLiZhu2021} for $\lambda, \sigma, \alpha \in \mathbb{R} > 0$. Here, $\lambda$ represents the drift strength, $\sigma$ the diffusion strength, and $\alpha$ the inverse temperature. 
\begin{subequations}\label{eqn:cbo-all}
\begin{equation}\label{eqn:cbo}
dX_t^i = -\lambda \cdot (X_t^i - m_{\alpha}(\rho_{X,t}^N)) \ dt + \sigma \cdot D(X_t^i - m_{\alpha}(\rho_{X,t}^N)) \ \mathrm{d}B_t^i.
\end{equation}
The vector $m_{\alpha}(\rho_{X,t}^N)$ is a weighted mean of particles whose weight function $\omega^{\alpha}_J = \exp(-\alpha \cdot J(x))$ permits the application of Laplace's principle:
\begin{equation}\label{eqn:laplace}
\begin{split}
m_{\alpha}(\rho) &:= \frac{1}{\lVert \omega^{\alpha}_J \rVert_{L^1(\rho)}} \int_{\mathbb{R}^d} x \cdot \omega^{\alpha}_J(x) \ d\rho, \qquad \rho_{X,t}^N = \frac{1}{N} \cdot \sum_{i=1}^{N} \delta_{X_t^i}
\end{split}
\end{equation}
\end{subequations}
Since $\text{supp}(\rho_{X,t}^N) = \{X_t^i\}_{i=1}^{N}$, it is clear by Laplace's principle that for $\alpha \gg 1$, the regularized global best $m_{\alpha}(\rho_{X,t}^N) \approx \overline{X}_t$, where
$$\overline{X}_t = \mathsf{argmin}\{J(X_t^1), J(X_t^2), \cdots, J(X_t^N)\}$$
Equation \eqref{eqn:cbo} is the finite-particle, \emph{anisotropic} model: the noise acts coordinatewise through $D(\cdot)$, which is the form inherited from the PSO velocity update \eqref{eqn:pso_sde}, and it is therefore the model we carry through the hierarchy of Figure \ref{fig:model_chart}. The convergence theory we invoke in Section~\ref{sec3} is that of \cite{carrillo2018analytical}, which is stated instead for the \emph{isotropic mean-field} model
\begin{equation}\label{eqn:cbo_iso}
\mathrm{d}\overline{X}_t = -\lambda \cdot \big(\overline{X}_t - m_{\alpha}(\rho_t)\big) \ \mathrm{d}t + \sigma \cdot \big|\overline{X}_t - m_{\alpha}(\rho_t)\big| \ \mathrm{d}B_t, \qquad \rho_t = \mathrm{Law}(\overline{X}_t),
\end{equation}
in which the diffusion coefficient is scalar and the empirical measure $\rho^N_{X,t}$ has been replaced by the mean-field law $\rho_t$. Working with \eqref{eqn:cbo_iso} is a matter of convenience rather than necessity: convergence statements for the anisotropic diffusion $D(\cdot)$ are also available in \cite{FornasierKlockRiedl2022} on the mean-field level, and \cite{CarrilloJinLiZhu2021} establishes a direct analogue of \cite[Theorem 4.1]{carrillo2018analytical} for that model. Both \eqref{eqn:cbo} and \eqref{eqn:cbo_iso} are \emph{first-order} systems, evolving the positions alone, whereas PSO \eqref{eqn:pso_sde} is second-order including momentum. Reduction (3) of Section~\ref{sec1} is what bridges the two: as the inertia $m \to 0$ the velocity equilibrates on a fast time scale and the second-order dynamics collapse onto a first-order system, a passage established rigorously as a zero-inertia limit in \cite{CiprianiHuangQiu2022}. Accordingly, Section~\ref{sec:cbo_first_order} and Section~\ref{sec3} refer to \eqref{eqn:cbo_iso}, with isotropic diffusion, whereas Section~\ref{sec:cbo_memory} and Section~\ref{sec4} focus on a second-order model with memory effect and anisotropic diffusion, more closely related to the PSO dynamics \eqref{eqn:pso_sde}. 
\newline
\newline
In the case of the CBO models \eqref{eqn:cbo-all} and \eqref{eqn:cbo_iso}, we have independent drift and diffusion strengths and no memory effect present. Convergence for the CBO dynamics is usually obtained via a two-step procedure (assuming that the objective $J$ admits a unique global minimizer $x^*$):
\begin{enumerate}
    \item $\textbf{Consensus Point Formation}$: For fixed $\alpha > 0$, as $t \rightarrow \infty$, $m_{\alpha}(\rho_t) \rightarrow \hat{x}_{\alpha}$, known as the consensus point since $\rho_t \rightarrow \delta_{\hat{x}_{\alpha}}$ as $t \rightarrow \infty$ for $\rho_t$ solving (\ref{eqn:cbo_iso}).
    \item $\textbf{Global Minimizer Approximation}$: As $\alpha \rightarrow \infty$, $\hat{x}_{\alpha} \rightarrow x^*$, the global minimizer by Laplace's principle.
\end{enumerate}
CBO with memory also retains a mean-field level description with independent drift and diffusion coefficients. We write down the CBO with memory equations as presented in \cite{Huang_2023} with parameters $m, \gamma, \lambda_1, \lambda_2, \sigma_1, \sigma_2, \kappa, \beta, \alpha, \theta \geq 0$ and $\gamma = 1-m \geq 0$:
\begin{subequations}\label{eqn:cbo_mem_all}
\begin{equation}\label{eqn:cbo_mem}
\begin{split}
dX^i_t &= V^i_t \ \text{d}t, \\
dY^i_t &= \kappa \cdot (X^i_t - Y^i_t) \cdot S^{\beta,\theta}(X^i_t, Y^i_t) \ \text{d}t, \\
m dV^i_t &= -\gamma V^i_t \ \text{d}t + \lambda_1 \cdot (Y^i_t - X^i_t) \ \text{d}t + \lambda_2 \cdot (m_{\alpha}(\hat{\rho}^{N}_{Y,t}) - X^i_t) \ \text{d}t \\ &+ \sigma_1 \cdot D(Y^i_t - X^i_t) \ \text{d}B^{1,i}_t + \sigma_2 \cdot D(m_{\alpha}(\hat{\rho}^{N}_{Y,t}) - X^i_t) \ \text{d}B^{2,i}_t. \\
\end{split}
\end{equation}
The CBO with memory model differs from PSO and CBO in two ways. The first is the replacement of the shifted Heaviside function in \eqref{eqn:time_cont} with a regularized Heaviside function $S^{\beta,\theta}(x,y)$ that converges to $S(x,y)$ as $\theta \rightarrow 0, \beta \rightarrow \infty$, and the second is that the weighted mean is defined in terms of the personal bests of the particles as opposed to the particles themselves.
\begin{equation}\label{eqn:cbo_mem_eq}
\begin{split}
S^{\beta,\theta}(x,y) &= 1 + \theta + \tanh(\beta \cdot (J(y) - J(x))) \qquad \hat{\rho}_{Y,t}^N = \frac{1}{N} \cdot \sum_{i=1}^{N} \delta_{Y_t^i}. \\
\end{split}
\end{equation}
\end{subequations}
To unify these models, the following visual shows the parameters of each algorithm and various reductions to recover PSO and CBO from the intermediary CBO with memory model. We note that to recover CBO from CBO with memory, one must have that $Y_0^i = X_0^i$ for all $n$ particles initially. Furthermore, the friction $\gamma = 1-m$ is common in the literature for second-order models. 
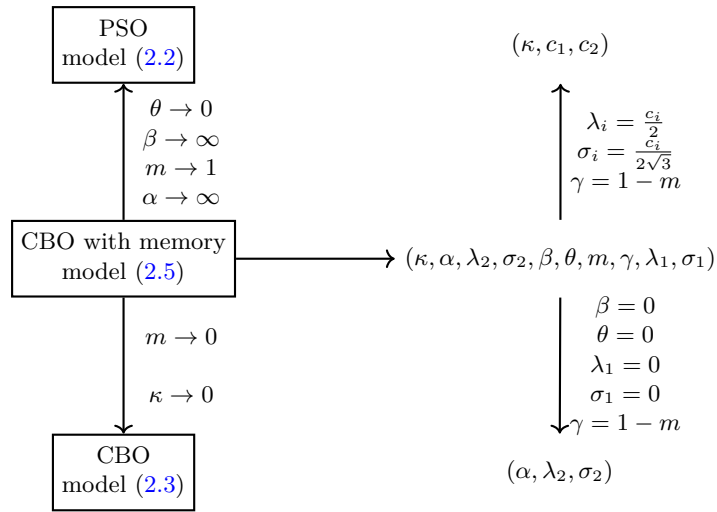
\begin{figure}[!ht]
\begin{center}
\begin{tikzpicture}[
    node distance=1.8cm and 2.3cm,
    every node/.style={rectangle, draw, minimum height=1cm, minimum width=1.5cm, align=center, font=\small},
    every path/.style={->, thick}
]

    \node (cbo) {CBO \\ model (\ref{eqn:cbo-all})};
    \node (cboMemory) [above=of cbo] {CBO with memory \\ model (\ref{eqn:cbo_mem_all})};
    \node (pso) [above=of cboMemory] {PSO \\ model (\ref{eqn:pso_sde_all})};

    \node[draw=none, right=2.1cm of cboMemory, align=left, font=\small] (params_cbo_memory) {$(\kappa, \alpha, \lambda_2, \sigma_2, \beta, \theta, m, \gamma, \lambda_1, \sigma_1)$};
    \node[draw=none, right=4.05cm of pso, align=left, font=\small] (params_pso) {$(\kappa, c_1, c_2)$};
    \node[draw=none, right=3.98cm of cbo, align=left, font=\small] (params_cbo) {$(\alpha, \lambda_2, \sigma_2)$};
    
    \draw[<-] (pso) -- (cboMemory) node[midway, right, draw=none] {$\theta \to 0$ \\ $\beta \rightarrow \infty$ \\ $m \rightarrow 1$ \\ $\alpha \rightarrow \infty$ };
    \draw[->] (cboMemory) -- (cbo) node[midway, right, draw=none] {$m \rightarrow 0$ \\ \\ $\kappa \to 0$};
    \draw[->] (params_cbo_memory) -- (params_pso) node[midway, right, draw=none] {$\lambda_i = \frac{c_i}{2}$ \\  $\sigma_i = \frac{c_i}{2\sqrt{3}}$ \\ 
    $\gamma = 1-m$};
    \draw[->] (params_cbo_memory) -- (params_cbo) node[midway, right, draw=none] {$\beta = 0$ \\ $\theta = 0$ \\ $\lambda_1 = 0$ \\ $\sigma_1 = 0$ 
    \\ $\gamma = 1-m$ };
    \draw[->] (cboMemory) -- (params_cbo_memory);
\end{tikzpicture}
\end{center}
\caption{Relations of parameters of PSO, CBO with memory, and CBO models}
\label{fig:model_chart}
\end{figure}
\subsection{Convergence Guarantees for First Order CBO}\label{sec:cbo_first_order}
A typical method for  studying interacting particle systems is to construct a suitable energy functional for the ensemble and show that it decays over time. For the first-order isotropic mean-field CBO dynamics \eqref{eqn:cbo_iso}, the relevant energy is the variance of the mean-field law $\rho_t$ around its own mean, defined by
\begin{equation}\label{eqn:cbo_variance}
V(\rho_t) := \frac{1}{2}\int_{\mathbb{R}^d} |x - \mathbb{E}(\rho_t)|^2 \ \mathrm{d}\rho_t, \qquad \mathbb{E}(\rho_t) := \int_{\mathbb{R}^d} x \ \mathrm{d}\rho_t.
\end{equation}
The idea, originally due to \cite{carrillo2018analytical}, is that if $V(\rho_t) \to 0$ exponentially fast, the mean-field law concentrates onto a Dirac mass $\delta_{x^*_\alpha}$ located at some consensus point $x^*_\alpha$. This decay is only guaranteed once the objective $J$ satisfies a regularity condition controlling how the weighted mean $m_\alpha(\rho_t)$ deviates from the bulk of the distribution. Following \cite{PinnauTotzeckTseMartin2017}, we assume the objective $J:\R^d\to\R$ satisfies 
\begin{enumerate}
    \item[(A1)] $\underline{J} := \inf J > 0$;
    \item[(A2)] $J \in \mathcal{C}^2(\mathbb{R}^d)$ with $\lVert \nabla^2 J \rVert_{\infty} \leq c_J$ for some constant $c_J > 0$.
\end{enumerate}
Condition (1) is a normalization that can always be satisfied via an additive shift of $J$, since it does not change the location of the minimizer. Condition (2), as the authors of \cite{carrillo2018analytical} note, may be enforced away from any bounded region without disturbing the location of the global minimizer, so it accommodates a rich class of nonconvex objectives with many spurious local minima (e.g.\ the Ackley and Rastrigin benchmarks).
\begin{remark}
The results in \cite[Section 4]{PinnauTotzeckTseMartin2017} rely on \cite[Assumption 4.1]{PinnauTotzeckTseMartin2017}, which additionally assumes $\Delta J \leq a_0 + a_1|\nabla J|^2$ in $\mathbb{R}^d$ for some parameters $a_0, a_1\ge 0$. We note that this condition follows from the Hessian bound $\lVert \nabla^2 J \rVert_{\infty} \leq c_J$ since 
$$\Delta J = \mathsf{tr}(\nabla^2 J) \leq d\lVert \nabla^2 J \rVert_{\infty} \leq dc_J$$
In fact, all results in \cite{PinnauTotzeckTseMartin2017} hold true only assuming (A1) above and choosing $c_0 = dc_J$, $c_1 = 0$ in \cite{PinnauTotzeckTseMartin2017}. The role of the Laplacian bound is that the well-posedness results in \cite[Sections 2-3]{PinnauTotzeckTseMartin2017} still hold in that case without the Hessian bound assumption. Since our work focues on the convergence results from \cite[Section 4]{PinnauTotzeckTseMartin2017}, which crucially rely on a Lipschitz control for the gradient of $J$, we will directly work with the Hessian bound here. Thus, we can apply \cite[Lemma 4.1]{PinnauTotzeckTseMartin2017} for any $\alpha > 0$. 
\end{remark}
\vspace{0.5em}
Next, we recall the main convergence result for the first-order CBO model \eqref{eqn:cbo_iso} that will form the basis of our analysis in Section~\ref{sec3}.
\begin{theorem}[{\cite[Theorem 4.1]{carrillo2018analytical}}]\label{thm:cbo_variance}
Let $J$ satisfy Assumptions (A1)-(A2), and let $\rho_t$ solve the isotropic mean-field CBO dynamics \eqref{eqn:cbo_iso}. Set $V_0 := V(\rho_0)$, $b_0 := \lVert \omega_J^\alpha \rVert_{L^1(\rho_0)}$, and
$$b_1 = b_1(\alpha, \lambda, \sigma) := 2\alpha e^{-2\alpha \underline{J}} \big(dc_J \sigma^2 + 2\lambda c_J\big).$$
Suppose the parameters $\alpha, \lambda, \sigma$ satisfy
$$b_1 < \frac{3}{4} \quad (\star), \qquad 2\lambda b_0^2 - V_0 - 2d\sigma^2 b_0 e^{-\alpha \underline{J}} \geq 0. \quad (\star\star)$$
Then $V(\rho_t) \leq V(\rho_0) e^{-qt}$ for all $t \geq 0$, with rate
$$q = 2\left(\lambda - \frac{d\sigma^2}{b_0} e^{-\alpha \underline{J}}\right) \geq V_0/b_0^2.$$
In particular, there exists a point $x^*_\alpha \in \mathbb{R}^d$, possibly depending on $\rho_0$, such that
$$\mathbb{E}(\rho_t) \rightarrow x^*_\alpha \qquad \text{and} \qquad m_\alpha(\rho_t) \rightarrow x^*_\alpha \qquad \text{as } t \rightarrow \infty.$$
\end{theorem}
\vspace{0.5em}
From Theorem \ref{thm:cbo_variance}, we have the desired exponential contraction of the variance on the level of the mean-field law. However, we remark that Theorem \ref{thm:cbo_variance} on its own only guarantees consensus formation at $x^*_\alpha$; a further application of the Laplace principle \cite[Theorem 4.2]{carrillo2018analytical} is needed to show that $x^*_\alpha$ can additionally be placed within any prescribed distance of the global minimizer $x^*$ by choosing $\alpha$ sufficiently large. This is exactly the aforementioned two-step procedure, in which one first takes $t$ large to achieve the consensus point, and then a further limit of the inverse temperature to infinity is needed to recover the global minimizer.

\subsection{Convergence Guarantees for CBO with Memory}\label{sec:cbo_memory}
Following \cite{Huang_2023}, the mean-field (McKean) description of \eqref{eqn:cbo_mem} is obtained by passing to the large particle limit $N \rightarrow \infty$, in which the empirical measure $\hat{\rho}_{Y,t}^N$ is replaced by the marginal law $\rho_{\overline{Y},t}$ of the process $\overline{Y}_t$ itself. This yields 
\begin{equation}\label{eqn:cbo_mem_mf}
\begin{split}
d\overline{X}_t &= \overline{V}_t \ \mathrm{d}t, \\
d\overline{Y}_t &= \kappa \cdot (\overline{X}_t - \overline{Y}_t) \cdot S^{\beta,\theta}(\overline{X}_t, \overline{Y}_t) \ \mathrm{d}t, \\
m \ d\overline{V}_t &= -\gamma \overline{V}_t \ \mathrm{d}t + \lambda_1 \cdot (\overline{Y}_t - \overline{X}_t) \ \mathrm{d}t + \lambda_2 \cdot \big(m_{\alpha}(\rho_{\overline{Y},t}) - \overline{X}_t\big) \ \mathrm{d}t \\
&\quad + \sigma_1 \cdot D(\overline{Y}_t - \overline{X}_t) \ \mathrm{d}B^{1}_t + \sigma_2 \cdot D\big(m_{\alpha}(\rho_{\overline{Y},t}) - \overline{X}_t\big) \ \mathrm{d}B^{2}_t.
\end{split}
\end{equation}
Here, the process is initialized according to $(\overline{X}_0, \overline{Y}_0, \overline{V}_0) \sim f_0$, $f_t = \mathrm{Law}(\overline{X}_t, \overline{Y}_t, \overline{V}_t)$, $\rho_{\overline{Y},t}$ denotes the marginal law of $\overline{Y}_t$ under $f_t$, and $m_\alpha(\rho_{\overline{Y},t})$ is defined exactly as in \eqref{eqn:laplace} with $\rho$ replaced by $\rho_{\overline{Y},t}$. The bar notation here denotes the fact that the particles are mean-field particles. 
\newline
\newline
We utilize the energy functional $\mathcal{H}(t)$ from \cite{Huang_2023}, which is written below. 
\begin{equation}\label{eqn:energy}
\begin{split}
\mathcal{H}(t) &= \left(\frac{\gamma}{2m}\right)^2 |\overline{X}_t - \mathbb{E}[\overline{X}_t]|^2 + \frac{3}{2}|\overline{V}_t|^2 + \frac{1}{2} \cdot \left(\frac{3\lambda_1}{m} + \frac{\gamma^2}{m^2}\right) \cdot |\overline{X}_t - \overline{Y}_t|^2 \\
&+ \frac{\gamma}{2m} \langle \overline{X}_t - \mathbb{E}[\overline{X}_t], \overline{V}_t \rangle + \frac{\gamma}{m} \langle \overline{X}_t - \overline{Y}_t, \overline{V}_t \rangle
\end{split}
\end{equation}
In \cite{Huang_2023}, it was shown that, under suitable conditions, $\mathbb{E}[\mathcal{H}(t)]$ decays over time. The first three terms in $\mathcal{H}(t)$ capture the variance, size of velocity, and deviation from the personal best. The last two terms are needed for a hypocoercivity approach: it is thanks to the mixing of position and velocity variables in the last two terms that decay of the first three terms is obtained. Convergence to the consensus point for CBO with memory relies on a regularity and tractability assumption on the objective $J : \mathbb{R}^d \rightarrow \mathbb{R}$, which we state in full below: 
\begin{enumerate}
    \item[(B1)] $\underline{J}:= \inf_{x \in \mathbb{R}^d} J(x)= 0$;
    \item[(B2)] $J \in \mathcal{C}^2(\mathbb{R}^d)$ with $\lVert \nabla^2 J \rVert_{\infty} \leq c_J$ for some constant $c_J > 0$;
    \item[(B3)] there exists a constant $L_J > 0$ such that
    $$|J(x) - J(x')| \leq L_J(|x|+|x'|)|x-x'| \qquad \forall \ x,x' \in \mathbb{R}^d;$$
    \item[(B4)] either $\overline{J} := \sup_{x \in \mathbb{R}^d} J(x) < \infty$, or there exist constants $C_J, R > 0$ such that
    $$J(x) - \underline{J} \geq C_J |x|^2 \qquad \forall \ x \in \mathbb{R}^d \text{ with } |x| \geq R.$$
\end{enumerate}
The first two assumptions are the same as (A1)-(A2) used for the first-order model. The additional assumptions are (A3) to account for the well-posedness of the model with the memory contributions, and (A4) purely for ease in the theoretical analysis. Under Assumptions (B1)-(B4), we recall \cite[Theorem 3.3]{Huang_2023}, which proves exponential convergence to the consensus point for CBO with memory effect. 
\newline 
\newline 
For (B1), we remark that $\underline{J} = 0$ is not a true restriction and mainly for convenience in the subsequent computations, since one can translate the objective function vertically without changing the location of the minimizer. We additionally note that this shift is not possible from the theorem inherited for the case of the first-order CBO model.
\begin{theorem}[{\cite[Theorem 3.3]{Huang_2023}}]\label{thm:cbo_mem_cv}
Let $J$ satisfy Assumptions (B1)-(B4), shift $J$ such that $\underline{J} = \inf J = 0$, and assume $0 < \theta < 2$. Moreover, assume the well preparation of the initial conditions $f_0$ in the sense that
\begin{itemize}
    \item[(H1)] $$\mu_1 := \frac{(\lambda_1 + 2\lambda_2) \cdot \gamma}{(2m)^2} - \left(\frac{9\lambda_2^2}{\gamma m} + \frac{3\sigma_2^2}{m^2} + \frac{3\lambda_1\gamma}{4m^2}\right) \cdot \frac{12}{\mathbb{E}_{f_0}[\exp(-\alpha \cdot J(\overline{Y}_0)]} > 0$$
    \item[(H2)] $$\mu_2 := \frac{(\lambda_1 + \lambda_2) \cdot \gamma}{m^2} + \kappa\theta \cdot \left(\frac{3\lambda_1}{m} + \frac{\gamma^2}{m^2}\right) - \frac{8\kappa^2\gamma}{m} - \frac{\lambda_2^2\gamma}{2m^2\lambda_1} - \frac{3\sigma_1^2}{2m^2} - \left(\frac{9\lambda_2^2}{\gamma m} + \frac{3\sigma_2^2}{m^2}\right)$$
    $$-\left(\frac{9\lambda_2^2}{\gamma m} + \frac{3\sigma_2^2}{m^2} + \frac{3\lambda_1\gamma}{(2m)^2} \right) \cdot \frac{24}{\mathbb{E}_{f_0}[\exp(-\alpha \cdot J(\overline{Y}_0)]} > 0$$
    \item[(H3)] $$\mu_3 := \left(\frac{\alpha \kappa m}{\lambda_1 \chi} \cdot (C_J + 2\alpha^2) + \frac{24C_J^2\kappa}{\alpha \chi^3}\right) \cdot \frac{\mathbb{E}_{f_0}[\mathcal{H}(0)]}{\mathbb{E}_{f_0}[\exp(-\alpha \cdot J(\overline{Y}_0))]} + \frac{6\kappa}{\alpha \chi} \cdot \frac{\mathbb{E}_{f_0}[|\nabla J(\overline{X}_0)|^2]}{\mathbb{E}_{f_0}[\exp(-\alpha \cdot J(\overline{Y}_0))]} < \frac{3}{32}$$
\end{itemize}
where
$$\chi = \frac{2}{5} \cdot \frac{\min\left\{\frac{\gamma}{2m}, \mu_1, \mu_2\right\}}{\left(\left(\frac{\gamma}{2m}\right)^2 + 1 + \frac{3\lambda_1}{m} + 2 \cdot \left(\frac{\gamma}{m}\right)^2\right)}$$
Then $\mathbb{E}[\mathcal{H}(t)]$ with $\mathcal{H}(t)$ as defined in Equation \eqref{eqn:energy} converges exponentially fast with rate $\chi$ to $0$ as $t \rightarrow \infty$. Moreover, there exists some consensus point $\hat{x}_{\alpha}$, which may depend on $\alpha$ and $f_0$ such that $\mathbb{E}[\overline{X}_t] \rightarrow \hat{x}_{\alpha}$ and $m_{\alpha}(\hat{\rho}_{Y,t}^N) \rightarrow \hat{x}_{\alpha}$ exponentially fast with rate $\chi/2$ as $t \rightarrow \infty$. Eventually, for any given accuracy $\epsilon > 0$, there exists $\alpha_0 > 0$ such that for all $\alpha > \alpha_0$, $\hat{x}_{\alpha}$ satisfies 
$$J(\hat{x}_{\alpha}) - \underline{J} \leq \epsilon.$$
\end{theorem}
\vspace{0.5em}
The conditions (H1)-(H3) on the model parameters and initial conditions are the result of a classical argument in stochastic analysis, which ensures that $\mathcal{H}(t)$ decays exponentially fast in expectation. These conditions are quite involved and non-linear in many of the parameters, which complicates the analysis of when these conditions are satisfied. 
\newline 
\newline
As shown earlier, this convergence result is for the time-continuous mean-field description of CBO with memory effect (\ref{eqn:cbo_mem_mf}). Furthermore, as shown in Figure \ref{fig:model_chart}, one can take the limits of some parameters to recover the PSO algorithm. Our main goal is to investigate the effect of the following coupling for $i = 1,2$ on the admissible set of coefficients:
\begin{equation}\label{eqn:coupling}
\lambda_i = \frac{c_i}{2}, \qquad \sigma_i = \frac{c_i}{2\sqrt{3}}.
\end{equation}
Here, $c_i$ correspond to the coupled drift--diffusion strength in the PSO algorithm (\ref{eqn:pso_sde}) and $(\lambda_i, \sigma_i)$ to the drift and diffusion strengths in the CBO dynamics with memory effect (\ref{eqn:cbo_mem_mf}). We analyze each statement in Theorem~\ref{thm:cbo_mem_cv} separately by substituting the PSO parameters for $\lambda_i, \sigma_i$ as defined in \eqref{eqn:coupling} and identifying when they still satisfy conditions (H1)-(H3). This will identify parameter regimes in which the existing theorem guarantees convergence of the regularized mean-field model under the PSO coupling, together with the corresponding decay rate.

\section{PSO Parameter Coupling in First Order CBO} \label{sec3}
We begin with the simpler setting, considering first order dynamics without memory effect and with isotropic diffusion. Following the variance-based analysis in \cite[Section 4]{carrillo2018analytical}, we would like to determine the conditions under which the PSO parameter coupling in \eqref{eqn:coupling} yields a non-empty set of admissible parameters for which the convergence guarantees in Theorem \ref{thm:cbo_variance} still hold. More precisely, we will investigate convergence properties of the isotropic mean-field CBO model (\ref{eqn:cbo_iso}) under the parameter coupling that appears in the PSO algorithm choosing $c_1 = 0$ (since there is no memory effect in (\ref{eqn:cbo_iso})), and where we couple $(\lambda, \sigma) = (\lambda_2, \sigma_2)$ via \eqref{eqn:coupling} for any $c_2 > 0$. Throughout this section, we assume $J$ satisfies (A1)-(A2). 
\begin{theorem}\label{thm:pso_in_cbo}
Let $J$ satisfy Assumptions (A1)-(A2), let $\rho_t$ solve the isotropic mean-field CBO dynamics \eqref{eqn:cbo_iso}, and write $V_0 := V(\rho_0)$ for the initial variance. There exists $V_0^* = V_0^*(\alpha, c_J, d, \rho_0) > 0$ and $\alpha^* = \alpha^*(V_0, c_J, d, J, \rho_0) > 0$ such that if either 
\begin{enumerate}[(i)]
    \item $\alpha > 0$ and $V_0 < V_0^*(\alpha)$, or
    \item $V_0 < 3/2d$ and $\alpha \in (0, \alpha^*(V_0))$,
\end{enumerate}
then the PSO coupling \eqref{eqn:coupling} $\lambda = c_2/2$ and $\sigma = c_2/2\sqrt{3}$ yields a non-empty set of admissible $c_2 > 0$ for which both $(\star)$ and $(\star\star)$ hold, and hence for which the conclusion of Theorem \ref{thm:cbo_variance} is valid.
\end{theorem}
\begin{proof}
For the first condition $(\star)$, we can solve as follows.
\begin{align*} 
2\alpha e^{-2\alpha \underline{J}} \cdot \left(c_Jd \left(\frac{c_2}{2\sqrt{3}}\right)^2 + 2 \cdot \frac{c_2}{2} \cdot c_J\right) < \frac{3}{4} &\implies 
c_Jd \cdot \frac{c_2^2}{12} + c_2 \cdot c_J < \frac{3e^{2\alpha \underline{J}}}{8\alpha} \\ 
&\implies
c_2^2 \cdot \frac{c_Jd}{12} + c_2 \cdot c_J - \frac{3e^{2\alpha \underline{J}}}{8\alpha} < 0.
\end{align*}
Solving for the roots of this upward-facing quadratic in $c_2$ (since $c_Jd > 0$) and simplifying,
\begin{align*}
c^*_{2,\pm} &= \frac{-c_J \pm \sqrt{c_J^2 - 4 \cdot \left(\frac{c_Jd}{12}\right) \cdot \left(\frac{-3e^{2\alpha \underline{J}}}{8\alpha}\right)}}{2 \cdot \left(\frac{c_Jd}{12}\right)} = \frac{6}{d}\left(\pm\sqrt{1 + \frac{d}{8c_J\alpha}\,e^{2\alpha \underline{J}}} - 1\right).
\end{align*}
Since the discriminant is positive, it holds that $c_{2,-}^* < 0$ and $c_{2,+}^* > 0$; write
\begin{equation}\label{eqn:c2plusstar}
c_{2,+}^*(\alpha) := \frac{6}{d}\left(\sqrt{1 + \frac{d}{8c_J\alpha}\,e^{2\alpha \underline{J}}} - 1\right) > 0,
\end{equation}
a quantity depending on $\alpha, c_J, d, \underline J$ but \emph{not} on $V_0$. The first condition $(\star)$ is satisfied exactly when $c_2 \in I_1(\alpha) := [0,c_{2,+}^*(\alpha))$.
\newline
\newline
We now substitute the coupling into the second condition $(\star\star)$ as follows.
\begin{align*}
2 \cdot \left(\frac{c_2}{2}\right)b_0^2 - V_0 - 2d\left(\frac{c_2}{2\sqrt{3}}\right)^2 b_0 \cdot e^{-\alpha \underline{J}} \geq 0 &\implies
c_2 \cdot b_0^2 - V_0 - \frac{c_2^2d}{6} \cdot b_0 \cdot e^{-\alpha \underline{J}} \geq 0 \\
&\implies c_2^2 \cdot \left(\frac{d\cdot b_0 \cdot e^{-\alpha \underline{J}}}{6}\right)  - c_2 \cdot b_0^2 + V_0 \leq 0.
\end{align*}
We once again use the quadratic formula since this is an upward facing quadratic.
\begin{align*}
c^{**}_{2,\pm}(\alpha, V_0) &= \frac{b_0^2 \pm \sqrt{b_0^4 - 4 \cdot \left(\frac{d\cdot b_0 \cdot e^{-\alpha \underline{J}}}{6}\right) \cdot V_0}}{2 \cdot \left(\frac{d\cdot b_0 \cdot e^{-\alpha \underline{J}}}{6}\right)}
= \frac{3b_0^2 \pm 3\sqrt{b_0^4 - \frac{2}{3} \cdot d\cdot b_0 \cdot e^{-\alpha \underline{J}} \cdot V_0}}{d\cdot b_0 \cdot e^{-\alpha \underline{J}}}.
\end{align*}
To ensure $(\star\star)$ holds, we need the quadratic to admit roots, and thus that $c^{**}_{2,\pm}$ is real. In other words, the discriminant must be non-negative, i.e.
\begin{equation}\label{eqn:V0star1}
V_0 \ \leq\ V_0^{(1)}(\alpha) := \frac{3b_0(\alpha)^3\,e^{\alpha \underline{J}}}{2d}, \tag{$\star\star\star$}
\end{equation}
in which case $c_{2,-}^{**}(\alpha,V_0), c_{2,+}^{**}(\alpha,V_0) > 0$ and the second condition $(\star\star)$ is satisfied exactly for $c_2 \in I_2(\alpha,V_0) := [c_{2,-}^{**}(\alpha,V_0),c_{2,+}^{**}(\alpha,V_0)]$.
\newline
\newline
Both regimes (i) and (ii) now reduce to showing $I_1(\alpha) \cap I_2(\alpha,V_0) \neq \emptyset$, i.e.\ $c_{2,-}^{**}(\alpha,V_0) < c_{2,+}^*(\alpha)$, together with $(\star\star\star)$. We establish each in turn.
\newline
\newline
\emph{Case (i): fixed $\alpha$, small $V_0$.} Fix $\alpha > 0$. On $[0, V_0^{(1)}(\alpha))$ the map
$$V_0 \ \longmapsto\ c_{2,-}^{**}(\alpha,V_0) = \frac{3b_0(\alpha)\,e^{\alpha \underline{J}}}{d}\left(1 - \sqrt{1 - \frac{2dV_0}{3b_0(\alpha)^3}e^{-\alpha \underline{J}}}\right)$$
is continuous, satisfies $c_{2,-}^{**}(\alpha,0) = 0$, and is strictly increasing, since
$$\frac{\partial}{\partial V_0}c_{2,-}^{**}(\alpha,V_0) 
= \left(b_0(\alpha)^2 \sqrt{1 - \frac{2dV_0}{3b_0(\alpha)^3}e^{-\alpha \underline{J}}}\right)^{-1}.$$
As $c_{2,+}^*(\alpha)$ in \eqref{eqn:c2plusstar} does not depend on $V_0$, and $c_{2,-}^{**}(\alpha,0) = 0 < c_{2,+}^*(\alpha)$, monotonicity and continuity give a unique 
$$V_0^*(\alpha) \in (0, V_0^{(1)}(\alpha)] \qquad \text{with} \qquad V_0 < V_0^*(\alpha) \ \Rightarrow \ c_{2,-}^{**}(\alpha,V_0) < c_{2,+}^*(\alpha)$$
Hence, for every $V_0 \in [0, V_0^*(\alpha))$, both $(\star\star\star)$ holds and $c_{2,-}^{**}(\alpha,V_0) < c_{2,+}^*(\alpha)$, so $I_1(\alpha)\cap I_2(\alpha,V_0) \neq \emptyset$ and $c_2 = c_{2,-}^{**}(\alpha,V_0)$ (say) is admissible for both $(\star)$ and $(\star\star)$. This proves (i).
\newline
\newline
\emph{Case (ii): fixed $V_0 < 3/(2d)$, small $\alpha$.} Writing both roots in a common form,
$$c_{2,+}^*(\alpha) = \frac{6}{d}\left(\sqrt{1 + \frac{d}{8c_J\alpha}e^{2\alpha \underline{J}}} - 1\right), \qquad c_{2,-}^{**}(\alpha,V_0) = \frac{3b_0(\alpha)\,e^{\alpha \underline{J}}}{d}\left(1 - \sqrt{1 - \frac{2dV_0}{3b_0(\alpha)^3}e^{-\alpha \underline{J}}}\right),$$
where $b_0(\alpha) = \lVert \omega_J^\alpha \rVert_{L^1(\rho_0)} = \int_{\mathbb{R}^d} e^{-\alpha J} \, \mathrm{d}\rho_0$ depends on $\alpha$ through the mean-field initial data $\rho_0$, and $\underline{J} > 0$ is a fixed constant determined by the objective. We resolve the comparison $c_{2,+}^*(\alpha) > c_{2,-}^{**}(\alpha,V_0)$ by analyzing both sides directly as $\alpha \rightarrow 0^+$, with $V_0$ now held fixed.
\newline
\newline
\emph{Behavior of $c_{2,+}^*$ as $\alpha \rightarrow 0^+$.} Since $\underline{J} > 0$ is fixed, the quantity
$$x(\alpha) := \frac{d}{8c_J\alpha}e^{2\alpha \underline{J}}$$
satisfies $x(\alpha) \rightarrow \infty$ as $\alpha \rightarrow 0^+$ (the factor $1/\alpha$ diverges while $e^{2\alpha \underline{J}} \rightarrow 1$). Using $\sqrt{1+x} - 1 \sim \sqrt{x}$ as $x \rightarrow \infty$,
$$c_{2,+}^*(\alpha) \sim \frac{6}{d}\sqrt{x(\alpha)} = \frac{6}{d}\sqrt{\frac{d}{8c_J\alpha}}\,e^{\alpha \underline{J}} \sim \frac{3}{\sqrt{2c_Jd\,\alpha}} \longrightarrow +\infty \qquad \text{as } \alpha \rightarrow 0^+.$$
\newline
\emph{Behavior of $c_{2,-}^{**}$ as $\alpha \rightarrow 0^+$.} Since $J \geq \underline{J} > 0$ is bounded below and continuous, dominated convergence gives $b_0(\alpha) \rightarrow 1$ as $\alpha \rightarrow 0^+$, and likewise $e^{\pm \alpha \underline{J}} \rightarrow 1$. Consequently,
$$\frac{2dV_0}{3b_0(\alpha)^3}e^{-\alpha \underline{J}} \longrightarrow \frac{2dV_0}{3} \qquad \text{as } \alpha \rightarrow 0^+.$$
By hypothesis $V_0 < 3/(2d)$, so this limit lies in $(0,1)$, and therefore
$$c_{2,-}^{**}(\alpha,V_0) \longrightarrow \frac{3}{d}\left(1 - \sqrt{1 - \frac{2dV_0}{3}}\right) =: L(V_0) < \infty \qquad \text{as } \alpha \rightarrow 0^+.$$
Since $c_{2,+}^*(\alpha) \rightarrow \infty$ while $c_{2,-}^{**}(\alpha,V_0) \rightarrow L(V_0) < \infty$ as $\alpha \rightarrow 0^+$, there exists $\alpha^*(V_0) > 0$ such that $c_{2,+}^*(\alpha) > c_{2,-}^{**}(\alpha,V_0)$ for all $\alpha \in (0,\alpha^*(V_0))$. Moreover, the hypothesis $V_0 < 3/(2d)$ is exactly what guarantees $(\star\star\star)$ to hold in this same limit, since as $\alpha \rightarrow 0^+$ its left-hand side $V_0^{(1)}(\alpha) \to 3/(2d) \cdot 1 > V_0$ while the right-hand side of the original comparison $e^{-\alpha \underline{J}} \rightarrow 1$; so $(\star\star\star)$ holds for all sufficiently small $\alpha$ as well. We shrink $\alpha^*(V_0)$ so that both hold simultaneously on $(0,\alpha^*(V_0))$, which implies $I_1(\alpha) \cap I_2(\alpha,V_0)$ is non-empty for every $\alpha \in (0,\alpha^*(V_0))$, so $c_2 = c_{2,-}^{**}(\alpha,V_0)$ (say) is admissible for both $(\star)$ and $(\star\star)$. This proves (ii).
\end{proof}
\begin{remark}\label{rem:alpha_large}
In fact, $V_0^*(\alpha) \rightarrow 0$ as $\alpha \rightarrow \infty$, and so Theorem \ref{thm:pso_in_cbo}(i) can be understood as an asymptotic convergence result for $\alpha > 0$ with a condition on the initial distribution being sufficiently concentrated. Further, Theorem \ref{thm:cbo_variance} only concerns consensus formation, for which $\alpha$ can be kept moderate, whereas driving the consensus point $x^*_\alpha$ close to the true global minimizer $x^*$ \cite[Theorem 4.2]{carrillo2018analytical} is the step that separately forces $\alpha \gg 1$. This motivates Theorem \ref{thm:pso_in_cbo}(ii), which allows more freedom on the choice of initial data $\rho_0$. 
\end{remark}
 
\section{PSO Parameter Coupling in CBO with Memory} 
\label{sec4}
The first-order model of Section~\ref{sec3} is derived from PSO only after three separate reductions: (1) the personal best is deleted ($c_1 = 0$), (2) the global best is smoothed, and (3) the inertia is sent to zero. To study the coupling \eqref{eqn:coupling} with both inertia and memory present, we now consider the second-order CBO model with memory \eqref{eqn:coupling}, in which both $c_1$ and $c_2$ are positive. In Sections \ref{sec41} - \ref{sec43}, we investigate each of the three parameter constraints (H1)-(H3) from Theorem \ref{thm:cbo_mem_cv}, respectively, and then combine those estimates in Section \ref{sec:proof_existence} to prove the existence of an admissible range for all parameters appearing in (\ref{eqn:cbo_mem_mf}).

\subsection{Statement (H1) for \texorpdfstring{$\mu_1$}{mu-1}}\label{sec41}
We begin by substituting the appropriate equations for $\lambda_i$ and $\sigma_i$ from \eqref{eqn:coupling} into $\mu_1$ to obtain
\begin{equation}\label{eqn:mu1}
\begin{split}
\mu_1 &:= \frac{(\frac{c_1}{2} + c_2) \cdot \gamma}{(2m)^2} - \left(\frac{9}{\gamma m} \cdot \left(\frac{c_2}{2}\right)^2 + \frac{3}{m^2} \cdot \left(\frac{c_2}{2\sqrt{3}}\right)^2 + \frac{3c_1\gamma}{8m^2}\right) \cdot K > 0,\\
K &:= \frac{12}{\mathbb{E}_{f_0}[\exp(-\alpha \cdot J(\overline{Y}_0)]} \geq 12,
\end{split}
\end{equation}
where $f_0$ denotes the initial law for the mean-field system (\ref{eqn:cbo_mem_mf}). Simplifying, we have 
$$\mu_1 = \frac{c_1\gamma}{8m^2} + \frac{c_2\gamma}{4m^2} - \left(\frac{9c_2^2}{4\gamma m} + \frac{c_2^2}{4m^2} + \frac{3c_1\gamma}{8m^2}\right) \cdot K > 0.$$
Multiplying by $8\gamma m^2$, we define $f(c_1, c_2)$ as 
\begin{align*}
f(c_1, c_2) &:= c_1\gamma^2 + 2c_2\gamma^2 - \left(18mc_2^2 + 2\gamma c_2^2 + 3c_1\gamma^2\right) \cdot K = 8\gamma m^2 \cdot \mu_1 > 0.
\end{align*}
Grouping together by $c_1$ and $c_2$, we have the following upon substituting $\gamma = 1-m$ with $m \in (0,1]$:
\begin{align*}
f(c_1,c_2) &= (-18mK - 2 \cdot (1-m) K) \cdot c_2^2 + 2 \cdot (1-m)^2 \cdot c_2 + ((1-m)^2 - 3K \cdot (1-m)^2) \cdot c_1 \\
&= -2K \cdot (8m + 1) \cdot c_2^2 + 2 \cdot (1-m)^2 \cdot c_2 - (1-m)^2 \cdot (3K-1) \cdot c_1.
\end{align*}
In order to ensure there is a solution to the inequality $f(c_1,c_2) > 0$, we check that the discriminant $\mathcal{D}_1$ of this quadratic in $c_2$ is greater than $0$.
\begin{align*}
\mathcal{D}_1 &= (2 \cdot (1-m)^2)^2 - 4 \cdot (-2K \cdot (8m + 1)) \cdot (-(1-m)^2 \cdot (3K-1) \cdot c_1) \\
&= 4 \cdot (1-m)^4 - 8K \cdot (8m+1) \cdot (1-m)^2 \cdot (3K - 1) \cdot c_1 \\
&= 4(1-m)^2 \cdot \left[(1-m)^2 - 2K \cdot (8m + 1) \cdot (3K - 1) \cdot c_1\right] > 0.
\end{align*}
Since $4(1-m)^2 > 0$ for $m \in (0,1)$, the expression in square brackets yields the following upper bound on $c_1$. 
$$c_1 < \frac{(1-m)^2}{2K \cdot (8m + 1) \cdot (3K - 1)} \equiv c_1^*(m,\alpha,J,f_0).$$
The upper bound is positive, which means that there is an admissible set for $c_1$ under this first condition, namely $c_1 \in (0,c_1^*)$. 
\newline 
\newline 
With a bound on $c_1$ to ensure that there are two roots of the quadratic in $c_2$, and the negative coefficient on the leading term to ensure that the parabola opens downward, we can compute the roots to get an admissible range for $c_2$. 
$$f(c_1, c_2) = -2K \cdot (8m+1) \cdot c_2^2 + 2 \cdot (1-m)^2 \cdot c_2 + (1-m)^2 \cdot (1-3K) \cdot c_1 = 0.$$
The quadratic formula then gives us the following values of $c_2$. 
\begin{align*}
c^*_{2,\pm}(c_1) &:= \frac{-2(1-m)^2 \mp \sqrt{(-2(1-m)^2)^2 - 4 \cdot (-2K) \cdot (8m+1) \cdot (1-m)^2 \cdot (1-3K) \cdot c_1}}{2 \cdot (-2K \cdot (8m+1))} \\
&= \frac{(1-m)^2 \pm (1-m) \sqrt{(1-m)^2 - 2K \cdot (8m+1) \cdot (3K-1) \cdot c_1}}{2K \cdot (8m+1)}.
\end{align*}
The above computations immediately yield the following result. 
\begin{lemma}\label{lem:mu1_ranges}
Under the parameter coupling \eqref{eqn:coupling}, condition (H1), i.e.\ $\mu_1 > 0$, is satisfied if and only if
$$0 < c_1 < c_1^* \qquad \text{and} \qquad 0 < c_{2,-}^*(c_1) < c_2 < c_{2,+}^*(c_1),$$
where
\begin{equation}\label{mu1_ranges}
\begin{split}
c_1^* &:= \frac{(1-m)^2}{2K \cdot (8m + 1) \cdot (3K - 1)}\,, \\
c_{2,\pm}^*(c_1) &:= \frac{(1-m)^2 \pm (1-m) \sqrt{(1-m)^2 - 2K \cdot (8m+1) \cdot (3K-1) \cdot c_1}}{2K \cdot (8m+1)} \,.
\end{split}
\end{equation}
\end{lemma}
 
\subsection{Statement (H2) for \texorpdfstring{$\mu_2$}{mu-2}}\label{sec42}
We again begin by substituting the appropriate equations for $\lambda_i$ and $\sigma_i$ into $\mu_2$ as follows, noting the presence of $K$ again.
\begin{align*}
\mu_2 &:= \frac{(\frac{c_1}{2} + \frac{c_2}{2}) \cdot \gamma}{m^2} + \kappa\theta \cdot \left(\frac{3c_1}{2m} + \frac{\gamma^2}{m^2}\right) - \frac{8\kappa^2\gamma}{m} - \frac{\left(\frac{c_2}{2}\right)^2\gamma}{2m^2 \cdot \frac{c_1}{2}} - \frac{3}{2m^2} \cdot \left(\frac{c_1}{2\sqrt{3}}\right)^2 \\
&- \left(\frac{9}{\gamma m} \cdot \left(\frac{c_2}{2}\right)^2 + \frac{3}{m^2} \cdot \left(\frac{c_2}{2\sqrt{3}}\right)^2\right) -\left(\frac{9}{\gamma m} \cdot \left(\frac{c_2}{2}\right)^2 + \frac{3}{m^2} \cdot \left(\frac{c_2}{2\sqrt{3}}\right)^2 + \frac{3c_1\gamma}{2(2m)^2} \right) \cdot 2K > 0.
\end{align*}
Simplifying, we have
\begin{align*}\mu_2 = \frac{c_1\gamma}{2m^2} + \frac{c_2\gamma}{2m^2} &+ \kappa\theta \cdot \left(\frac{3c_1}{2m} + \frac{\gamma^2}{m^2}\right) - \frac{8\kappa^2\gamma}{m} - \frac{c_2^2\gamma}{4m^2 \cdot c_1} - \frac{c_1^2}{8m^2} \\
&- \left(\frac{9c_2^2}{4\gamma m} + \frac{c_2^2}{4m^2}\right) -\left(\frac{9c_2^2}{2\gamma m} + \frac{c_2^2}{2m^2} + \frac{3c_1\gamma}{4m^2} \right) \cdot K > 0.
\end{align*}
Scaling by $8\gamma c_1m^2$ across, we define $u(c_1,c_2)$ as follows.
\begin{align*}
u(c_1,c_2) &:= -u_2(c_1) \cdot c_2^2 + u_1(c_1) \cdot c_2 + u_0(c_1) = 8\gamma c_1 m^2 \cdot \mu_2,
\end{align*}
where we group based on the degree of $c_2$, substitute $\gamma = 1-m$, and define
\begin{equation}\label{eqn:u_coeffs}
\begin{split}
u_2(c_1) &= 2 \cdot \left(\gamma^2 + c_1 \cdot (1+2K) \cdot (9m + \gamma)\right),\\ 
u_1(c_1) &= 4c_1 \gamma^2, \\ 
u_0(c_1) &= -2c_1^2\gamma^2(3K-2) + \kappa\theta \gamma \cdot \left(12c_1^2m + 8 \cdot \gamma^2c_1\right) - 64\kappa^2 \gamma^2c_1m - \gamma c_1^3.\\ 
\end{split}
\end{equation}
We can then compute the two roots of the quadratic in $c_2$ of $u(c_1,c_2)$, with $u_0, u_1, u_2$ defined above.
$$c^{**}_{2,\pm} = \frac{u_1 \pm \sqrt{u_1^2 + 4u_0u_2}}{2u_2}$$
To have two real roots, we need the discriminant $\mathcal{D}_2(c_1) := u_1(c_1)^2 + 4u_0(c_1)u_2(c_1)$ to be positive. Since $u_2(c_1) > 0$ and $u_1(c_1) > 0$ for all choices of parameters, it is sufficient to exhibit parameters for which $u_0(c_1) \geq 0$: in that case $\mathcal{D}_2(c_1) \geq u_1(c_1)^2 > 0$ automatically. To this end, we rewrite
\begin{align*}
u_0(c_1) & := -2c_1^2\gamma^2(3K-2) + \kappa\theta\gamma \cdot \left(12c_1^2m + 8 \cdot \gamma^2c_1\right) - 64\kappa^2 \cdot \gamma^2c_1m - \gamma c_1^3 \\
&= c_1\gamma \cdot \left[-2c_1\gamma(3K-2) + \kappa\theta \cdot \left(12c_1m + 8 \cdot \gamma^2\right) - 64\kappa^2 \cdot \gamma m - c_1^2\right] \\
&= c_1\gamma \cdot \left[-c_1^2 + 2c_1 \cdot (-3 \gamma K + 2\gamma + 6\kappa\theta m) + 8\kappa\theta \gamma^2 - 64\kappa^2\gamma m\right] 
\end{align*}
Since $c_1$ and $\gamma$ are both positive, it suffices to determine the roots of the remaining quadratic in $c_1$ to ensure $u_0(c_1) \geq 0$. Computing the roots, we have 
\begin{align*}
c^{**}_{1,\pm} &= \frac{2(2\gamma + 6m\kappa \theta - 3\gamma K) \pm \sqrt{(2(2\gamma + 6m\kappa \theta - 3\gamma K))^2 - 4 \cdot (-1) \cdot (8\gamma^2 \kappa \theta - 64\kappa^2 \gamma m)}}{2} \\
&= (2\gamma + 6m\kappa \theta - 3\gamma K) \pm \sqrt{(2\gamma + 6m\kappa \theta - 3\gamma K)^2 + (8\gamma^2 \kappa \theta - 64\kappa^2 \gamma m)}.
\end{align*}
In order to find a positive $c_1$ such that the inequality $u_0(c_1) \geq 0$ holds, it is sufficient to have
$(8\gamma^2 \kappa \theta - 64\kappa^2 \gamma m) > 0$. Solving for $\theta$, we have 
\begin{equation}\label{mu2_pos}
\theta > \frac{8\kappa m}{\gamma}.
\end{equation}
The condition ~\eqref{mu2_pos} ensures that there is a
positive $c_1$ with $u_0(c_1) \ge 0$ regardless of the sign of
$2\gamma+6m\kappa\theta-3\gamma K$, and we remark that it is a necessary
condition for a positive root when $2\gamma+6m\kappa\theta-3\gamma K<0$. Condition ~\eqref{mu2_pos} also implies
that $c_{1,-}^{**}<0$ in this regime. We conclude that $u_0(c_1)\geq0$
if $c_1\in(0,c_{1,+}^{**}]$, but stress that
Equation~\eqref{mu2_pos} is not sharp for $\mu_2>0$ to hold. In \cite{Huang_2023}, the authors assumed that $0\leq\theta\ll1$, as $S^{\beta,\theta}$ converges to the shifted Heaviside function appearing in the PSO model \eqref{eqn:pso_sde_all} when $\theta\to0$ and $\beta\to\infty$, and their proof of (H2) additionally requires $\theta<2$. 
\newline 
\newline 
Condition \eqref{mu2_pos} therefore has to be read together with that upper bound: the admissible window $8\kappa m/\gamma < \theta < 2$ is non-empty only for $\kappa < \gamma/(4m)$, so the memory strength $\kappa$ must be kept small relative to $\gamma/m$. Condition \eqref{mu2_pos} therefore pushes $\theta$ away from the regime in which the memory-regularized model approximates PSO unless $\kappa$ is taken sufficiently small. It is worth noting that \eqref{mu2_pos} is a \emph{sufficient} condition for $u_0(c_1)\ge0$ at small $c_1$, so its failure does not by itself rule out admissible parameters. However, the obstruction is present at the endpoint $\theta = 0$, and can be read from $\mu_2$.
\begin{proposition}\label{prop:theta_zero}
Under the coupling \eqref{eqn:coupling} with $\theta = 0$, no choice of $c_1, c_2 > 0$ satisfies (H2).
\end{proposition}
\begin{proof}
Setting $\theta = 0$ and discarding the non-positive terms from the simplified $\mu_2$ above, we have
\begin{align*} 
\mu_2 &\leq \frac{\gamma c_1}{2m^2} +\frac{\gamma c_2}{2m^2}
-\frac{\gamma c_2^2}{4m^2c_1} -\frac{3K\gamma c_1}{4m^2} \\ 
&\leq \frac{(2\gamma + 3K\gamma)c_1}{4m^2} + \frac{\gamma c_2}{2m^2} - \frac{\gamma c_2^2}{4m^2c_1}
\end{align*}
The middle two terms form a downward parabola in $c_2$ maximized at $c_2 = c_1$, whence
\[
\frac{\gamma c_2}{2m^2} -\frac{\gamma c_2^2}{4m^2c_1}
\leq \frac{\gamma c_1}{4m^2} \qquad \text{for all } c_2 > 0,
\]
and therefore reach a contradiction since since $K \geq 12$ by \eqref{eqn:mu1}.
\[
\mu_2 \leq \frac{\gamma c_1}{2m^2} + \frac{\gamma c_1}{4m^2} - \frac{3K\gamma c_1}{4m^2} = \frac{3\gamma c_1}{4m^2}\,(1-K) < 0
\]
\end{proof}
With the above constraint on $\theta$, we can state the admissible ranges for $\mu_2$.
\begin{lemma}\label{lem:mu2_ranges}
Fix $c_1 > 0$, $0 < \theta < 2$ and let $u_0, u_1, u_2$ be as in \eqref{eqn:u_coeffs}, so that $u_1(c_1) > 0$ and $u_2(c_1) > 0$. Set $\mathcal{D}_2(c_1) := u_1(c_1)^2 + 4u_0(c_1)u_2(c_1)$ and
\begin{equation}\label{mu2_ranges}
c_{2,\pm}^{**}(c_1) := \frac{u_1(c_1) \pm \sqrt{\mathcal{D}_2(c_1)}}{2u_2(c_1)}\,.
\end{equation}
Under the parameter coupling \eqref{eqn:coupling}, there exists $c_2 > 0$ with $\mu_2 > 0$ if and only if $\mathcal{D}_2(c_1) > 0$, and in that case
$$\{c_2 > 0 : \mu_2(c_1, c_2) > 0\} = \mathcal{I}_2(c_1) := \big(\max\{0,\,c_{2,-}^{**}(c_1)\},\ c_{2,+}^{**}(c_1)\big).$$
Moreover $c_{2,-}^{**}(c_1) \leq 0$ if and only if $u_0(c_1) \geq 0$; thus the left endpoint of $\mathcal{I}_2(c_1)$ is $0$ exactly when $u_0(c_1) \geq 0$, and is the strictly positive root $c_{2,-}^{**}(c_1)$ in the remaining regime $u_0(c_1) < 0 < \mathcal{D}_2(c_1)$. Finally, if $\theta > 8\kappa m/\gamma$, then $u_0(c_1) \geq 0$ (hence $\mathcal{D}_2(c_1) > 0$) for every $c_1 \in (0, c_{1,+}^{**}]$, where
$$c_{1,+}^{**}:= (2\gamma + 6m\kappa \theta - 3\gamma K) + \sqrt{(2\gamma + 6m\kappa \theta - 3\gamma K)^2 + (8\gamma^2 \kappa \theta - 64\kappa^2 \gamma m)}\,.$$
\end{lemma}
\begin{proof}
Since $8\gamma c_1m^2 > 0$, we have $\mu_2 > 0 \iff u(c_1,c_2) > 0$, and $u(c_1,\cdot)$ is a downward parabola. It is positive somewhere iff $\mathcal{D}_2(c_1)>0$, and thus exactly on $(c_{2,-}^{**}, c_{2,+}^{**})$; intersecting with $\{c_2>0\}$ gives $\mathcal{I}_2(c_1)$, which is non-empty because $c_{2,+}^{**} = (u_1 + \sqrt{\mathcal{D}_2})/(2u_2) > 0$. For the sign of the smaller root, $u(c_1,0) = u_0(c_1)$, so $u_0(c_1) \ge 0$ iff $0$ lies between the roots, i.e.\ iff $c_{2,-}^{**} \le 0$. Sufficient conditions for the last claim to hold are $\theta > 8\kappa m/\gamma$ and $c_1 \in (0,c_{1,+}^{**}]$, as shown by the computation preceding the lemma.
\end{proof}
Next, we need to ensure that the two solution ranges for $c_1$ and $c_2$ that arise from $\mu_1$ and $\mu_2$ are not disjoint. 
\begin{itemize}
    \item  We recall that $\mu_1 > 0$ requires $0 < c_1 < c_1^*$, and that $\mu_2 > 0$ requires $\mathcal{D}_2(c_1) > 0$, which by Lemma \ref{lem:mu2_ranges} holds in particular for $0 < c_1 \leq c_{1,+}^{**}$ when $\theta > 8\kappa m/\gamma$. If we define $\overline{c}_1 := \min\{c_1^*, c_{1,+}^{**}\}$, then $c_1 \in (0,\overline{c}_1)$ together with $8\kappa m/\gamma < \theta < 2$ is sufficient for both.
    \item By Lemma \ref{lem:mu1_ranges}, $\mu_1 > 0$ requires the two-sided range $c_2 \in \mathcal{I}_1(c_1) := (c_{2,-}^*(c_1), c_{2,+}^*(c_1))$, and $c_{2,-}^*(c_1) > 0$ for every $c_1 \in (0,c_1^*)$ (it vanishes only at $c_1 = 0$ and is increasing in $c_1$). Since $\mu_2 > 0$ additionally requires $c_2 \in \mathcal{I}_2(c_1)$, the two ranges intersect precisely when $\max\{c_{2,-}^*(c_1), c_{2,-}^{**}(c_1)\} < \min\{c_{2,+}^*(c_1), c_{2,+}^{**}(c_1)\}$, and $\mathcal{I}(c_1) := \mathcal{I}_1(c_1)\cap\mathcal{I}_2(c_1)$ is then exactly the set of $c_2$ satisfying both (H1) and (H2). 
\end{itemize}
\subsection{Statement (H3) for \texorpdfstring{$\mu_3$}{mu-3}}\label{sec43}
This statement has two parts: one concerning the convergence rate of CBO with memory, and the other providing an upper bound on certain quantities to guarantee convergence of CBO with memory. Writing $\mu_3$ and the convergence rate $\chi$ in terms of the PSO parameters $c_1, c_2$, we have 
\begin{align*}
\chi &= \frac{2}{5} \cdot \frac{\min\left\{\frac{\gamma}{2m}, \frac{f(c_1,c_2)}{8\gamma m^2}, \frac{u(c_1,c_2)}{8\gamma c_1 m^2}\right\}}{\left(\left(\frac{\gamma}{2m}\right)^2 + 1 + \frac{3c_1}{2m} + 2 \cdot \left(\frac{\gamma}{m}\right)^2\right)}, \\
\mu_3 &:= \kappa \cdot \left(\left(\frac{2\alpha m}{c_1\chi} \cdot (C_{J} + 2\alpha^2) + \frac{24C_{J}^2}{\alpha \chi^3}\right) \cdot \frac{\mathbb{E}_{f_0}[\mathcal{H}(0)] \cdot K}{12} + \frac{6}{\alpha \chi} \cdot \frac{\mathbb{E}_{f_0}[|\nabla J(\overline{X}_0)|^2] \cdot K}{12}\right) < \frac{3}{32}.
\end{align*}
While these statements are very complicated in their current form, we will use asymptotic arguments to show that (H3) can be satisfied by taking appropriate limits. In a similar spirit to that of \cite{carrillo2018analytical}, we will need to take some ``well-preparedness'' conditions on the initial distribution. 

\subsection{Proof of Existence of Admissible Parameters} \label{sec:proof_existence}
In this section we prove that the set of parameters satisfying (H1)--(H3) from Theorem \ref{thm:cbo_mem_cv} is non-empty under the PSO coupling \eqref{eqn:coupling}. We do this in two stages. Theorem \ref{thm:existence_mu1_mu2} treats the initial law $f_0$ as fixed and gives (H1)--(H2) unconditionally, and in addition, guarantees (H3) conditionally on control of the initial energy and initial gradient of the objective function. Theorem \ref{thm:existence_family} then shows that all assumptions can be satisfied by exhibiting a one-parameter family of initial laws along which those bounds are met while the dependence on $K$ and $f_0$ is tracked. 
\begin{theorem}[Admissible Parameters for Fixed Initial Law]\label{thm:existence_mu1_mu2}
Let \(m\in(0,1)\) and \(\gamma:=1-m>0\). Fix \(\alpha>0\), \(\kappa>0\), and let \(K \ge 12\) be any constant; for a given initial law $f_0$ the relevant value is $K = 12/\mathbb{E}_{f_0}\!\left[e^{-\alpha J(\overline{Y}_0)}\right]$.
Assume $\kappa < \gamma/(4m)$ and that the memory-regularization parameter satisfies
\begin{equation}\label{eq:theta_cond_existence}
\frac{8\kappa m}{\gamma} < \theta < 2 .
\end{equation}
Then:
\begin{enumerate}
    \item[(a)] There exists \(\varepsilon = \varepsilon(m,\kappa,\theta,K)>0\) such that for every \(c_1\in(0,\varepsilon)\) the set
    \(\mathcal{I}(c_1) = \mathcal{I}_1(c_1)\cap\mathcal{I}_2(c_1) \subset \mathbb{R}_{>0}\) of \(c_2\)-values for which
    \(\mu_1>0\) and \(\mu_2>0\) hold simultaneously under the coupling \eqref{eqn:coupling} is non-empty and open.
    \item[(b)] Let $f_0$ be an initial law with the given value of $K$. For any \(c_1\in(0,\varepsilon)\) and \(c_2\in\mathcal{I}(c_1)\) as in (a), write \(\chi=\chi(c_1,c_2)>0\)
    for the resulting rate. If, in addition, \(f_0\) satisfies
    \begin{equation}\label{eq:mu3_wellprep}
    \mathbb{E}_{f_0}[\mathcal{H}(0)] < \eta_1(c_1,c_2) \qquad \text{and} \qquad \mathbb{E}_{f_0}[|\nabla J(\overline{X}_0)|^2] < \eta_2(c_1,c_2),
    \end{equation}
    where
    \[
    \eta_1(c_1,c_2) := \frac{9}{16\kappa K}\left(\frac{2\alpha m}{c_1\chi}(C_J+2\alpha^2) + \frac{24C_J^2}{\alpha\chi^3}\right)^{-1},
    \qquad
    \eta_2(c_1,c_2) := \frac{3\alpha\chi}{32\kappa K},
    \]
    then \(\mu_3 < 3/32\) also holds, so that (H1)-(H3) are simultaneously satisfied for this \(f_0\).
\end{enumerate}
\end{theorem} 
\begin{proof}
By Lemma \ref{lem:mu1_ranges}, $\mu_1>0$ holds if and only if $c_1\in(0,c_1^*)$ and $c_2\in \mathcal{I}_1(c_1) =(c_{2,-}^*(c_1),c_{2,+}^*(c_1))\subset(0,\infty)$, and one checks from the closed form in \eqref{mu1_ranges} that
\begin{equation}\label{eq:mu1_lower_scaling_fixed}
c_{2,-}^*(c_1) = \frac{3K-1}{2}\,c_1 + \mathcal{O}(c_1^2) \qquad \text{as } c_1\to 0^+,
\end{equation}
since we have the following by a first-order Taylor expansion.
$$\sqrt{\gamma^2 - 2K(8m+1)(3K-1)c_1} = \gamma - K(8m+1)(3K-1)c_1/\gamma + \mathcal{O}(c_1^2)$$ 
By Lemma \ref{lem:mu2_ranges}, under \eqref{eq:theta_cond_existence} the quantity $8\gamma^2\kappa\theta - 64\kappa^2\gamma m = 8\kappa\gamma(\theta\gamma - 8\kappa m)$ is strictly positive, so $u_0(c_1)>0$ for all sufficiently small $c_1>0$; hence $\mathcal{D}_2(c_1)>0$, $c_{2,-}^{**}(c_1)<0$, and $\mathcal{I}_2(c_1)=(0,c_{2,+}^{**}(c_1))$ for such $c_1$. As $c_1\to0^+$ we have $u_1(c_1)=\mathcal{O}(c_1)$, $u_2(c_1)\to 2\gamma^2$, and $u_0(c_1) = 8\kappa\gamma^2(\theta\gamma-8\kappa m)c_1+\mathcal{O}(c_1^2)$, so $\mathcal{D}_2(c_1) = u_1^2(c_1)+4u_0(c_1)u_2(c_1) = \Theta(c_1)$, and therefore
\begin{equation}\label{eq:mu2_upper_scaling_fixed}
c_{2,+}^{**}(c_1)=\Theta(\sqrt{c_1})\qquad \text{as } c_1\to 0^+ .
\end{equation}
We then have $\mathcal{I}(c_1)=\mathcal{I}_1(c_1)\cap\mathcal{I}_2(c_1) = \bigl(c_{2,-}^*(c_1),\,\min\{c_{2,+}^*(c_1),c_{2,+}^{**}(c_1)\}\bigr)$ for all sufficiently small $c_1$ (with $c_{2+}(c_1)$ being the smaller quantity in the limit).
By \eqref{eq:mu1_lower_scaling_fixed} and \eqref{eq:mu2_upper_scaling_fixed}, there exist constants \(C_1,C_2>0\) and \(\varepsilon\in(0,c_1^*)\) such that for all \(0<c_1<\varepsilon\),
\[
c_{2,-}^*(c_1)\le C_1 c_1,
\qquad
c_{2,+}^{**}(c_1)\ge C_2\sqrt{c_1}.
\]
Since \(c_1<C_2^2/C_1^2\) eventually forces \(C_1c_1<C_2\sqrt{c_1}\), shrinking \(\varepsilon\) further if necessary gives \(c_{2,-}^*(c_1) < c_{2,+}^{**}(c_1)\) for all \(0<c_1<\varepsilon\). Combined with $c_{2,-}^*(c_1)<c_{2,+}^*(c_1)$ (Lemma \ref{lem:mu1_ranges}), this shows \(\mathcal{I}(c_1)\) is non-empty and open for all \(c_1\in(0,\varepsilon)\), establishing part (a).
\newline
\newline
For part (b), fix \(c_1\in(0,\varepsilon)\) and \(c_2\in\mathcal{I}(c_1)\) as furnished by part (a), and let \(\chi=\chi(c_1,c_2)>0\) denote the corresponding rate. With \(c_1,c_2,\chi,\kappa,\alpha,m,K,C_J\) all fixed, $\mu_3$ is an explicit non-negative linear combination
\begin{equation}\label{eqn:mu3-compact}
    \mu_3 = \Xi_1(c_1,c_2)\cdot\mathbb{E}_{f_0}[\mathcal{H}(0)] + \Xi_2(c_1,c_2)\cdot\mathbb{E}_{f_0}[|\nabla J(\overline{X}_0)|^2],
\end{equation}
where
\begin{equation}\label{eqn:Xi}
\Xi_1(c_1,c_2) := \frac{\kappa K}{12}\left(\frac{2\alpha m}{c_1\chi}(C_J+2\alpha^2) + \frac{24C_J^2}{\alpha\chi^3}\right) > 0,
\qquad
\Xi_2(c_1,c_2) := \frac{\kappa K}{2\alpha\chi} > 0\,.
\end{equation}
Requiring each of the two terms to be below \(3/64\) is sufficient for \(\mu_3<3/32\), which yields exactly the thresholds \(\eta_1(c_1,c_2) = \tfrac{3}{64\Xi_1(c_1,c_2)}\) and \(\eta_2(c_1,c_2)=\tfrac{3}{64\Xi_2(c_1,c_2)}\) stated in \eqref{eq:mu3_wellprep}. Since \(\Xi_1,\Xi_2>0\) are finite for the fixed \((c_1,c_2)\), the thresholds \(\eta_1,\eta_2\) are strictly positive, and \eqref{eq:mu3_wellprep} is exactly the assertion that \(f_0\) meets them. This establishes part (b).
\end{proof}
Next, we show that there exist choices of parameters $(c_1, c_2)$ and initial laws $f_0$ for which the conditions in $(\textbf{a})$ and $(\textbf{b})$ of Theorem \ref{thm:existence_mu1_mu2} hold. In other words, the set of parameters and initial data admissible for Theorem \ref{thm:cbo_mem_cv} under the coupling \eqref{eqn:coupling} is non-empty.

\begin{theorem}[Existence of Admissible Parameters and Initial Law]\label{thm:existence_family}
Let $J$ satisfy (B1)--(B4), normalized so that $\underline{J}=0$, and let $x^*$ be the global minimizer. Let $m\in(0,1)$, $\gamma := 1-m$, and fix $\alpha>0$, $\kappa \in (0, \gamma/(4m))$ and $\theta \in \big(8\kappa m/\gamma,\, 2\big)$. Let $Z$ be a centered, compactly supported, non-degenerate random vector on $\mathbb{R}^d$, with $\mathrm{supp}(Z)\subset \overline{B_R(0)}$ and $\mathbb{E}|Z|^2 \in (0,\infty)$, and for $\delta>0$ define the initial law $f_0^\delta := \mathrm{Law}(\overline{X}_0^\delta, \overline{Y}_0^\delta, \overline{V}_0^\delta)$ by
\begin{equation}\label{eqn:delta_family}
\overline{X}_0^\delta = \overline{Y}_0^\delta = x^* + \delta Z, \qquad \overline{V}_0^\delta = 0,
\end{equation}
and write $K_\delta := 12\big/\mathbb{E}_{f_0^\delta}\![e^{-\alpha J(\overline{Y}_0^\delta)}]$. Then there exist $c_1, c_2 > 0$ and $\delta_0 > 0$ such that for every $\delta \in (0,\delta_0)$ the parameters $(c_1,c_2)$ and the initial law $f_0^\delta$ satisfy (H1)-(H3) simultaneously. 
\end{theorem}

\begin{proof}
To investigate Condition (H3) for the initial condition $f_0^{\delta}$ for small enough $\delta$, we split our argument into five steps. 
\newline 
\newline
\emph{Step 1: $K_\delta \to 12$ at rate $\mathcal{O}(\delta^2)$ as $\delta \rightarrow 0^+$.} By (B1) and (B2) we have $J(x^*)=0$, $\nabla J(x^*)=0$ and $\lVert\nabla^2 J\rVert_\infty \le c_J$, so Taylor's theorem gives $0 \le J(x^*+\delta z) \le \tfrac12 c_J \delta^2 |z|^2$ for all $z$. Since $|Z| \le R$ a.s, the following bound implies $12 \le K_\delta \le 12\,e^{\alpha c_J R^2 \delta^2/2} = 12 + \mathcal{O}(\delta^2)$ as $\delta \to 0^+$.
$$1 \ \ge\ \mathbb{E}_{f_0^\delta}\!\left[e^{-\alpha J(\overline{Y}^\delta_0)}\right] \ \ge\ e^{-\alpha c_J R^2\delta^2/2},$$
\emph{Step 2: The initial-data quantities are $\mathcal{O}(\delta^2)$.} Since $\overline{V}_0^\delta = 0$ and $\overline{X}_0^\delta = \overline{Y}_0^\delta$, every term of $\mathcal{H}(0)$ in \eqref{eqn:energy} vanishes except the first, and since $Z$ is centered, $\mathbb{E}[\overline{X}_0^\delta] = x^*$. Hence
$$\mathbb{E}_{f_0^\delta}[\mathcal{H}(0)] = \left(\frac{\gamma}{2m}\right)^2 \mathbb{E}\big|\overline{X}_0^\delta - \mathbb{E}[\overline{X}_0^\delta]\big|^2 = \left(\frac{\gamma}{2m}\right)^2 \mathbb{E}|Z|^2\,\delta^2 .$$
Likewise $|\nabla J(x^*+\delta z)| = |\nabla J(x^*+\delta z) - \nabla J(x^*)| \le c_J \delta|z|$, so
$$\mathbb{E}_{f_0^\delta}\big[|\nabla J(\overline{X}_0^\delta)|^2\big] \le c_J^2\,\mathbb{E}|Z|^2\,\delta^2 .$$
\emph{Step 3: Choice of $(c_1,c_2)$ at $K=12$.} Apply Theorem \ref{thm:existence_mu1_mu2}(a) with the constant $K=12$ (since (a) is an algebraic statement about the parameter $K \ge 12$ and does not refer to any initial law). This produces $c_1 \in (0,\varepsilon)$ and a non-empty open set $\mathcal{I}(c_1) \ni c_2$ on which $\mu_1|_{K=12}>0$ and $\mu_2|_{K=12}>0$, and hence $\chi|_{K=12}>0$. Fix such a pair $(c_1,c_2)$ for the remainder of the proof.
\newline
\newline
\emph{Step 4: Stability in $K$.} For fixed $(c_1,c_2,m,\kappa,\theta)$, both $\mu_1$ and $\mu_2$ are affine functions of $K$, and $\chi$ is a continuous function of $(\mu_1,\mu_2)$. Since $\mu_1,\mu_2 > 0$ at $K=12$, and $\mu_1, \mu_2, \chi$ are decreasing functions of $K$, there exists $\overline{K}>12$ such that $\mu_1 > 0$, $\mu_2 > 0$ for all $K \in [12,\overline{K}]$; on this compact interval $\chi$ is continuous and strictly positive, so
$$\chi_* := \min_{K \in [12,\overline{K}]} \chi(K) > 0 .$$
\emph{Step 5: (H3) for small $\delta$.} For any $K \in [12,\overline{K}]$ the coefficients of \eqref{eqn:Xi} obey the $\delta$-independent bounds
$$\Xi_1 \le \overline{\Xi}_1 := \frac{\kappa \overline{K}}{12}\left(\frac{2\alpha m}{c_1\chi_*}(C_J+2\alpha^2) + \frac{24C_J^2}{\alpha\chi_*^3}\right), \qquad \Xi_2 \le \overline{\Xi}_2 := \frac{\kappa \overline{K}}{2\alpha\chi_*},$$
both finite. By Step 1 there is $\delta_1>0$ with $K_\delta \in [12,\overline{K}]$ for all $\delta \in (0,\delta_1)$, so that (H1) and (H2) hold along the family for such $\delta$. By Step 2 and \eqref{eqn:mu3-compact},
$$\mu_3 \ \le\ \overline{\Xi}_1 \left(\frac{\gamma}{2m}\right)^2\mathbb{E}|Z|^2\,\delta^2 + \overline{\Xi}_2\, c_J^2\,\mathbb{E}|Z|^2\,\delta^2 \ =\ \mathcal{O}(\delta^2).$$
Choosing $\delta_0 \in (0,\delta_1)$ small enough that this bound is below $3/32$ gives (H3) for all $\delta \in (0,\delta_0)$.
\end{proof}

\begin{remark}\label{rem:delta_local}
The family \eqref{eqn:delta_family} is precisely a swarm initialized at rest in a $\delta$-neighborhood of the global minimizer, with every particle already at its own personal best. Theorem \ref{thm:existence_family} is therefore an existence statement of a local character: it shows the admissible set in Theorem~\ref{thm:existence_mu1_mu2} is non-empty, but the conditions on the initial law mean that $f_0$ needs to be already very concentrated around a point. This is the same tension quantified numerically in Section~\ref{sec6}, and is not an artifact of the construction: the smallness of $\chi$ forced by the $(\mu_1,\mu_2)$ wedge is what makes the admissible $\delta$ in Theorem~\ref{thm:existence_family} so small. The same phenomenon is already visible for the simpler first-order model \eqref{eqn:cbo_iso}: Theorem \ref{thm:pso_in_cbo}(i) constrains the spread of $\rho_0$, through $V_0=V(\rho_0)$, also see Remark~\ref{rem:alpha_large}.
\end{remark}

\section{Tradeoff and Gap Analysis} \label{sec5}
We have explicitly derived the admissible $(c_1,c_2)$ ranges under the PSO coupling, which are coupled non-linearly through $m,\alpha,\theta,\kappa$. In the standard first-order CBO model, the drift $\lambda$ and diffusion $\sigma$ can be tuned independently. The PSO reduction instead fixes $\lambda_i,\sigma_i\propto c_i$, so a single parameter $c_i$ simultaneously controls both the pull toward the current best point and the noise around it; increasing $c_i$ to strengthen the drift necessarily strengthens the diffusion by the same factor, in contrast to CBO where the two can be balanced independently. We examine the resulting tradeoffs for all model parameters. 
\newline
\newline 
Notice that $\mu_1$ can be written as
    \begin{equation*}
    \mu_1 = \frac{-c_1\gamma(3K-1)}{8m^2} + \frac{1}{4m}\left[\frac{\gamma c_2}{m} - \frac{9Kc_2^2}{\gamma} - \mathbf{\frac{Kc_2^2}{m}}\right]
    \end{equation*}
with $K$ as in (\ref{eqn:mu1}). The last term in bold comes from $\sigma_2$ (via $3\sigma_2^2/m^2$ in the original statement of (H1)). Because the coupling forces $\sigma_2\propto c_2$, this is a term that is \emph{quadratic and negative} in $c_2$, competing against the single term linear and positive in $c_2$ ($\gamma c_2/m$, from the drift $\lambda_2$). Increasing $c_2$ without bound eventually makes the negative quadratic terms dominate, forcing $\mu_1$ below zero. This is exactly why Lemma \ref{lem:mu1_ranges} gives a bounded interval $(c_{2,-}^*(c_1), c_{2,+}^*(c_1))$ for $c_2$, rather than unbounded. 
\newline 
\newline 
The main effective constraint on $\mu_2$ is instead the memory-regularization parameter $\theta$: by Proposition \ref{prop:theta_zero}, $\theta=0$ admits no parameters at all, and \eqref{mu2_pos} is the sufficient condition under which we can exhibit conditions (H1)-(H2). The parameter $\kappa$ enters through the same expression, and we discuss it below.
\newline
\newline
The limits required to recover PSO from CBO with memory ($\alpha\to\infty$, $m\to1$; see Figure \ref{fig:model_chart}) are precisely those that shrink the admissible ranges derived in Section~\ref{sec4} to nothing. We make this precise for $c_1^*(m,K)$ below, also see overview in Table \ref{table:limits}.
$$c_1^*(m,K) := \frac{(1-m)^2}{2K \cdot (8m + 1) \cdot (3K - 1)}$$
\textbf{Inertia limit ($m$).} Differentiating the closed form for $c_1^*$ from Lemma \ref{lem:mu1_ranges} with $K$ held fixed,
\begin{equation*}
\frac{\partial c_1^{*}(m,K)}{\partial m} = \frac{(m-1)(4m+5)}{K(3K-1)(8m+1)^2} \leq 0 \qquad \text{for } m\in(0,1),
\end{equation*}
so $c_1^*(\cdot,K)$ is non-increasing in $m$, with $$c_1^*(0,K) = \dfrac{1}{2K(3K-1)}>0$$ 
and $c_1^*(m,K)\to0$ as $m\to1^-$.
\newline
\newline
\textbf{Annealing limit ($\alpha$).} As $\alpha\to0^+$, $K\to12^+$ exactly (since $\mathbb{E}_{f_0}[e^{-\alpha J(\overline{Y}_0)}]\to1^-$), giving the finite positive limit 
$$c_1^*(m,12^+) = \dfrac{(1-m)^2}{840(8m+1)}.$$
As $\alpha\to\infty$, provided the initial law $f_0$ places no mass exactly at a global minimizer of $J$ (a generic, non-atomic condition), so $K\to\infty$ by dominated convergence; since $c_1^*\propto 1/(K(3K-1))=O(1/K^2)$, the range collapses to $0$.

\begin{table}[h!]
\centering
\begin{tabular}{|c|c|c|c|c|}
 \hline
 Limit & $m \to 0$ & $m \to 1$ & $\alpha \to 0^+$ & $\alpha \to \infty$ \\
 \hline
 $c_1^*(m,\alpha)$ & $\dfrac{1}{2K(3K-1)}$ & $0$ & $\dfrac{(1-m)^2}{840(8m+1)}$ & $0$ \\
 \hline
\end{tabular}
\caption{Both limits needed to recover PSO ($m\to1$ and $\alpha\to\infty$) drive $c_1^*$ to $0$, i.e.\ the admissible set for $(c_1, c_2)$ vanishes.}
\label{table:limits}
\end{table}
\begin{table}[!ht]
\centering
\begin{tabular}{|c|c|c|c|c|}
 \hline
 Parameter & Description & $\mu_1$ & $\mu_2$ & $\mu_3$ \\
 \hline
 $\alpha$ & Annealing term & $\downarrow$ & $\downarrow$ & $\downarrow$ (via $K$) \\
 $\theta$ & Memory buffer & $-$ & $\uparrow$ & $\uparrow$ \\
 $\kappa$ & Memory effect & $-$ & n.m.\ (see below) & n.m.\ (see below) \\
 $m$ & Inertia term & $\downarrow$ & $\downarrow$ (via $c_{1,+}^{**}$) & $\downarrow$ \\
 \hline
\end{tabular}
\caption{
Sensitivity of the admissible $(c_1,c_2)$ range to each parameter, within the admissible window \eqref{eq:theta_cond_existence} for $(\kappa,\theta)$. For each parameter and each condition $\mu_1, \mu_2, \mu_3$, the arrow shows the direction in which that parameter should be moved to enlarge the admissible range for $c_1, c_2$: $\downarrow$ marks a parameter whose increase shrinks the range (so it should be decreased to widen it); $\uparrow$ marks a parameter whose increase enlarges the range.
A dash ``$-$'' indicates the parameter does not enter that condition; ``n.m.'' marks relationships shown or suspected to be non-monotonic (see discussion below).}
\label{table:sensitivity}
\end{table}
The $\alpha$ and $m$ rows for $\mu_1$ follow directly from Lemma \ref{lem:mu1_ranges} and the computation above. The $\theta$ column deserves one caveat: $\theta$ does not enter $\mu_1$ or $\mu_3$ explicitly, and helps $\mu_2$ and (through $\chi$) $\mu_3$ only within the window \eqref{eq:theta_cond_existence}. Two further entries are more subtle:
\begin{itemize}
    \item \emph{$\kappa$'s effect on $\mu_2$ is non-monotonic.} Differentiating $c_{1,+}^{**}$ (Lemma \ref{lem:mu2_ranges}) with respect to $\kappa$ and evaluating numerically over the valid range $0<\kappa<\gamma\theta/(8m)$ enforced by \eqref{mu2_pos} shows $c_{1,+}^{**}$ first \emph{increases}, reaches an interior maximum, and then \emph{decreases} back toward $0$ as $\kappa$ approaches the boundary $\gamma\theta/(8m)$. This is consistent with the structure of $u_0(c_1)$: the term $8\kappa\theta\gamma^2$ is linear and helps, while $-64\kappa^2\gamma m$ is quadratic and eventually dominates. A small increase in $\kappa$ (relative to $\theta$) is therefore beneficial for $\mu_2$, but pushing $\kappa$ too close to the bound in \eqref{mu2_pos} is counterproductive.
\item \emph{$m$'s effect on $\mu_2$ is monotone through the sufficient bound $c_{1,+}^{**}$.} The sufficient
condition~\eqref{mu2_pos} that certifies $u_0(c_1) \ge 0$ reads $8\kappa m < \theta(1-m)$, so it is
available only for $m \in (0, m^\dagger(\kappa,\theta))$, where $m^\dagger(\kappa,\theta) := \frac{\theta}{\theta+8\kappa}$
is the value at which~\eqref{mu2_pos} holds with equality;
once $m \ge m^\dagger$, this route no longer certifies $u_0(c_1) \ge 0$ for any $c_1 > 0$.
On the remaining range the certified interval itself shrinks as the inertia grows: by
Lemma~\ref{lem:c1plus_monotone} below, $c_{1,+}^{**}$ is strictly decreasing in $m$ on $(0,m^\dagger)$,
the same direction as the effect of $m$ on $c_1^*$ recorded in Lemma~\ref{lem:mu1_ranges}. We stress
that this concerns the sufficient upper bound $c_{1,+}^{**}$ only; the monotonicity in $m$ of the full
region $\{\mu_2 > 0\}$ would require a separate argument.
\end{itemize}
\begin{lemma}\label{lem:c1plus_monotone}
Fix $K \ge 12$, $\kappa > 0$ and $0 < \theta < 2$, and let $c_{1,+}^{**}$ be as in
Lemma~\ref{lem:mu2_ranges}. Then $m \mapsto c_{1,+}^{**}$ is strictly decreasing on
$(0, m^\dagger(\kappa,\theta))$, where $m^\dagger(\kappa,\theta) = \theta/(\theta+8\kappa)$.
\end{lemma}
\begin{proof}
By the computation preceding Lemma~\ref{lem:mu2_ranges}, $c_{1,+}^{**}$ is the larger root of
$c_1^2 - 2Ac_1 - B = 0$ with $A := \gamma(2-3K) + 6m\kappa\theta$ and
$B := 8\kappa\gamma(\theta\gamma - 8\kappa m)$. Substitute $s := m/\gamma = m/(1-m)$, a strictly
increasing bijection of $(0,1)$ onto $(0,\infty)$ under which $\gamma = 1/(1+s)$ and \eqref{mu2_pos}
becomes $s < \theta/(8\kappa)$, i.e.\ $m < m^\dagger$. Setting $h := c_{1,+}^{**}/\gamma$ and dividing
the quadratic by $\gamma^2$, $h = h(s)$ is the larger root of
\begin{equation}\label{eqn:h_quadratic}
P_s(h) := h^2 + 2h\big[(3K-2) - 6\kappa\theta s\big] - 8\kappa\theta + 64\kappa^2 s = 0 .
\end{equation}
For $s < \theta/(8\kappa)$ we have $P_s(0) = 64\kappa^2 s - 8\kappa\theta < 0$, so
\eqref{eqn:h_quadratic} has exactly one positive root, and
\begin{equation}\label{eqn:h_denom}
h(s) + (3K-2) - 6\kappa\theta s
= \sqrt{\big[(3K-2) - 6\kappa\theta s\big]^2 + 8\kappa\theta - 64\kappa^2 s} > 0 .
\end{equation}
Next, $h(s) < 16\kappa/(3\theta)$. Indeed, evaluating $P_s$ at $h_c := 16\kappa/(3\theta)$, the two
terms linear in $s$ cancel and
$$P_s(h_c) = 8\kappa\left(\frac{32\kappa}{9\theta^2} + \frac{4(3K-2)}{3\theta} - \theta\right) > 0 ,$$
since $K \ge 12$ and $0 < \theta < 2$ give $4(3K-2)/(3\theta) \ge 136/(3\theta) > 2 > \theta$; as
$P_s$ is an upward parabola with $P_s(0) < 0$, its larger root satisfies $h(s) < h_c$.
By \eqref{eqn:h_denom} the partial derivative $\partial_h P_s$ does not vanish at $h(s)$, so $h$ is
differentiable and implicit differentiation of \eqref{eqn:h_quadratic} gives
$$h'(s) = \frac{12\kappa\theta\,h(s) - 64\kappa^2}{2\big[h(s) + (3K-2) - 6\kappa\theta s\big]}
= \frac{12\kappa\theta\big(h(s) - \tfrac{16\kappa}{3\theta}\big)}{2\big[h(s) + (3K-2) - 6\kappa\theta s\big]} < 0 ,$$
the numerator being negative by $h(s) < h_c$ and the denominator positive by \eqref{eqn:h_denom}.
Finally $c_{1,+}^{**} = \gamma h(s) = h(s)/(1+s)$, so
$$\frac{\mathrm{d}}{\mathrm{d}s}\,c_{1,+}^{**} = \frac{(1+s)h'(s) - h(s)}{(1+s)^2} < 0 ,$$
using $h'(s) < 0$ and $h(s) > 0$. Since $s$ is strictly increasing in $m$, the claim follows.
\end{proof}
The preceding estimates from Theorem \ref{thm:cbo_mem_cv} show that the admissible coefficient ranges shrink in the limits used to recover PSO, under the assumptions therein, including small $c_1, c_2$, initial laws concentrated near the global minimizer, and $m$ bounded away from $1$. These restrictions limit the applicability of the convergence guarantees to swarms with broadly distributed initial positions, which are more commonly used in practice. We return to the practical implications of this gap in Sections~\ref{sec6}--\ref{sec7}. 
 
\section{Numerical Results} \label{sec6}
To illustrate the shape and size of the admissible set, we perform a numerical sweep of the parameter space $(c_1, c_2)$ using the derived convergence guarantees in the case of the PSO parameter coupling. We seek to identify the regime where the convergence conditions of Theorem \ref{thm:cbo_mem_cv} are simultaneously satisfied.
\newline 
\newline
We fix $m=0.5, \alpha=2.0$, and $\kappa=0.05$. To satisfy the memory buffer requirement derived in Eq. \eqref{mu2_pos}, we set $\theta = 0.5$ and $K=12.1$; we also fix the Hessian bound $c_J = 1$ in (B2). Figure \ref{fig:admissible_grid} illustrates the intersection of the resulting conditions as stated in Theorem \ref{thm:existence_mu1_mu2}, and we remark that the scale of the axes is on the order of $\mathcal{O}(10^{-3})$. 
\begin{figure}[!htbp]
\centering
\includegraphics[width=0.75\textwidth]{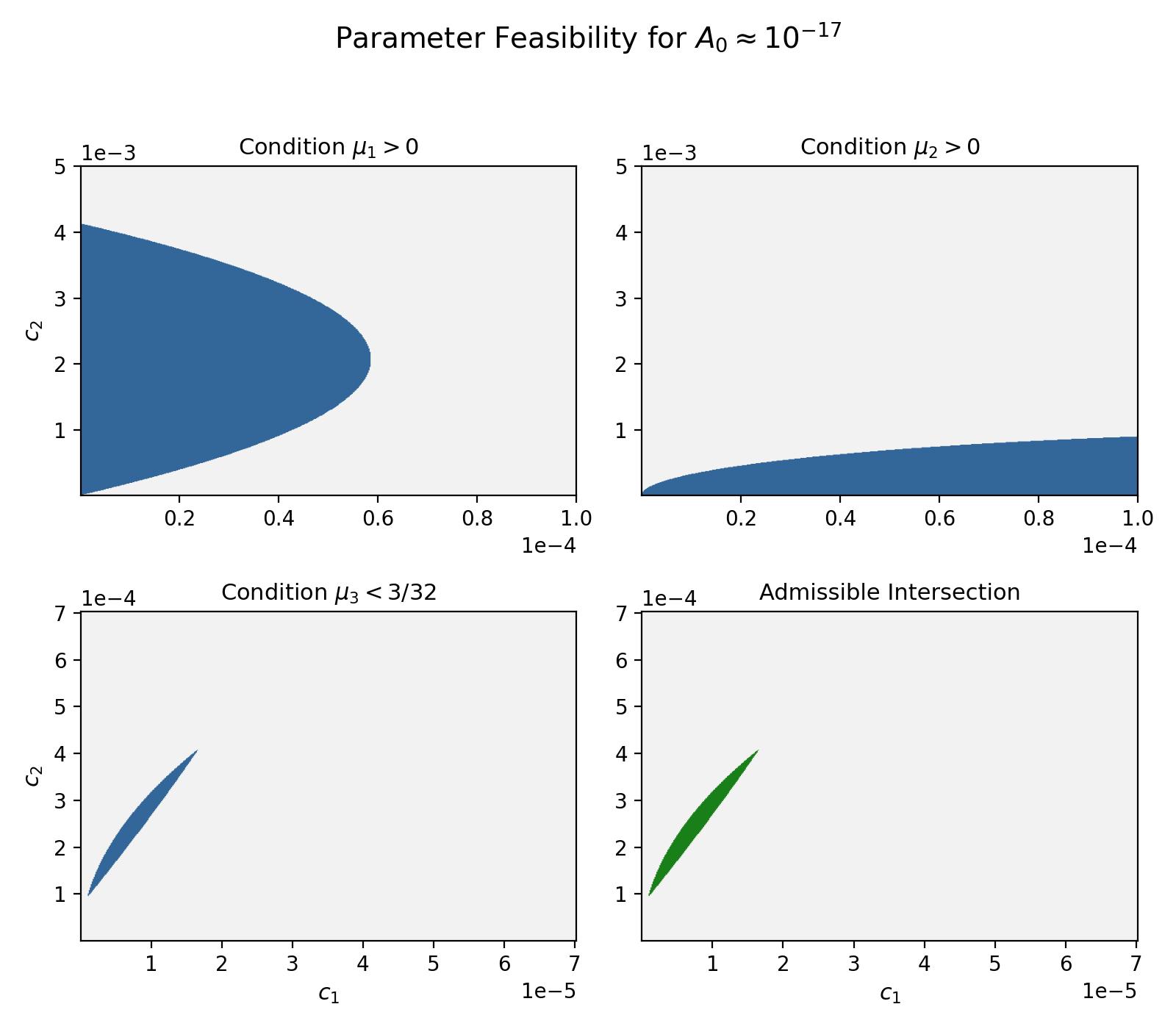}
\caption{Admissible parameter regions for PSO coupling. The sweep is conditional on the prescribed summary values $K = 12.1$, $A_0 = 10^{-17}$, $B_0 = 10^{-6}$ and $c_J = 1$, which are treated here as free parameters rather than computed from a specific objective and initial law. Panels (a), (b), and (c) show $\{\mu_1>0\}$, $\{\mu_2>0\}$ and $\{\mu_3 > 0\}$ respectively; panel (d) shows the fully admissible set for the prescribed budgets $A_0=10^{-17}$, $B_0=10^{-6}$. The intersection of $\mu_1$ (lower bound) and $\mu_2$ (upper bound) forms a wedge, while the energy condition $\mu_3$ restricts the feasible set further, to near-singular values of $c_1$.}
\label{fig:admissible_grid}
\end{figure}
\newline
\newline
We emphasize that this is a \emph{conditional algebraic feasibility check}, as the three quantities
$$K = \frac{12}{\mathbb{E}_{f_0}[e^{-\alpha J(\overline{Y}_0)}]}, \qquad A_0 = \frac{K}{12}\mathbb{E}_{f_0}[\mathcal{H}(0)], \qquad B_0 = \frac{K}{12}\mathbb{E}_{f_0}\big[|\nabla J(\overline{X}_0)|^2\big]$$
are all functionals of the \emph{same} pair $(J, f_0)$, but are prescribed here as independent summary parameters, with no objective and no initial law specified. What the sweep therefore establishes is: \emph{if} a pair $(J,f_0)$ realizes these three values, then the displayed $(c_1,c_2)$ region is exactly the set on which (H1)--(H3) hold. It does not by itself establish that such a pair exists. That gap is closed separately by Theorem \ref{thm:existence_family}, in which $K_\delta$, $\mathbb{E}_{f_0^{\delta}}[\mathcal{H}(0)]$ and $\mathbb{E}_{f_0^{\delta}}[|\nabla J(\overline{X}_0)|^2]$ are all computed from a specific law $f_0^\delta$ and are consequently tied to one another. 
\newline
\newline
The numerical results illustrate that while a non-empty admissible set exists, it is highly constrained. Concretely, for each fixed $c_1$, the two conditions act on $c_2$ from opposite sides of the interval $\mathcal{I}(c_1)=\mathcal{I}_1(c_1)\cap\mathcal{I}_2(c_1)$ established in Theorem \ref{thm:existence_mu1_mu2}:
\begin{itemize}
    \item $\mu_1 > 0$ enforces the \emph{lower} endpoint $c_2 > c_{2,-}^*(c_1)$ from Lemma \ref{lem:mu1_ranges}. Below this threshold, the linear drift term $\gamma c_2/m$ in $\mu_1$ (from $\lambda_2$) is too weak relative to the friction and $K$-weighted terms it must overcome, so the hypocoercivity argument underlying Theorem \ref{thm:cbo_mem_cv} no longer guarantees that the particles are pulled toward the weighted mean $m_\alpha(\rho_{Y,t})$ fast enough for $\mathbb{E}[\mathcal{H}(t)]$ to decay.
    \item $\mu_2 > 0$ enforces the \emph{upper} endpoint $c_2 < c_{2,+}^{**}(c_1)$ from Lemma \ref{lem:mu2_ranges}. Above this threshold, the coupling $\sigma_2\propto c_2$ (discussed in Section~\ref{sec5}) makes the diffusion terms, which scale quadratically in $c_2$, outweigh the drift terms; the sign conditions used to bound $\chi$ from below then fail, so the theorem no longer guarantees that the variance and velocity components of $\mathcal{H}(t)$ contract, rather than grow, along the dynamics.
\end{itemize}
In the colored region in Figure~\ref{fig:admissible_grid}, $\theta > 8\kappa m/\gamma$ holds, so $u_0(c_1)>0$ and the lower endpoint $c_{2,-}^{**}(c_1)$ of Lemma \ref{lem:mu2_ranges} is negative; the second bullet is therefore the only constraint $\mu_2$ imposes here. 
Figure \ref{fig:admissible_grid} visualizes $\mathcal{I}(c_1)$ as the wedge-shaped region enforced by these two opposing bounds as $c_1$ varies, narrowing as $c_1$ approaches $\overline{c}_1$ (defined in Section \ref{sec42}) and vanishing beyond it. The top row of the figure shows this directly: the individual regions $\mu_1>0$ and $\mu_2>0$ each occupy a large portion of the $(c_1,c_2)$-grid on their own, and it is only their intersection, $\mathcal{I}(c_1)$, that is confined to the thin wedge described above. The bottom row brings the third condition into view, and as a result, the fully admissible set where all three conditions hold simultaneously is far smaller than $\mathcal{I}(c_1)$ alone would suggest.
\newline
\newline
A critical finding of our numerical sweep is the role of the initial energy constants
$$A_0 := \tfrac{K}{12}\mathbb{E}_{f_0}[\mathcal{H}(0)] < \tfrac{K}{12}\eta_1(c_1,c_2) \qquad B_0 := \tfrac{K}{12}\mathbb{E}_{f_0}[|\nabla J(\overline{X}_0)|^2] < \tfrac{K}{12}\eta_2(c_1,c_2),$$
the $K$-rescaled versions of \eqref{eq:mu3_wellprep} in Theorem~\ref{thm:existence_mu1_mu2}. Because $\eta_1(c_1,c_2)$ carries a factor of $\chi^{-3}$ alongside $(c_1\chi)^{-1}$, and both $c_1$ and $\chi$ are themselves forced to be small by the $\mu_1,\mu_2$ wedge documented in Section~\ref{sec:proof_existence}, the ceiling $\frac{K}{12}\eta_1(c_1,c_2)$ (the largest $A_0$ any point of the wedge can tolerate) collapses very rapidly as $c_1\to0$. Maximizing this ceiling over the wedge at the parameters above, we find
$$\sup_{(c_1,c_2)\,\in\,\mathcal{I}} \tfrac{K}{12}\eta_1(c_1,c_2) = 3.61\times10^{-17}, \quad \text{attained at } c_1 = 6.27\times10^{-6},\ c_2 = 2.52\times10^{-4},\ \chi = 7.75\times10^{-6},$$
which is why Figure~\ref{fig:admissible_grid} prescribes $A_0=10^{-17}$. At this same point, the ceiling on the gradient term is $\frac{K}{12}\eta_2 = 2.40\times10^{-6}$, comfortably above the prescribed $B_0=10^{-6}$. Since both quantities shrink at the same rate, it is always the $A_0$ threshold that is reached first as $\delta\to0$, and the gradient budget $B_0$ is met automatically, with room to spare, at whatever $\delta$ the $A_0$ constraint already forces. In this sense, $B_0$  plays no role in setting how concentrated the initial swarm must be. With the above choices of parameters, the largest admissible $A_0$ over the wedge is $5.49\times10^{-18}$, $3.64\times10^{-17}$, $1.87\times10^{-15}$ and $6.53\times10^{-15}$ at $\theta = 0.45, 0.5, 1, 2$. We can see that the gain is real but saturates at the boundary $\theta = 2$, beyond which the constants in (H2) are no longer those of Theorem \ref{thm:cbo_mem_cv}. 
\newline 
\newline
Recalling that $\mathcal{H}(0)$ itself measures the initial variance, velocity, and deviation from the personal best (see \eqref{eqn:energy}), this is consistent with the tension identified in Section~\ref{sec5}: the present sufficient conditions certify decay only for swarms that are already close to consensus. What this does \emph{not} identify is a mechanism by which the classical algorithm escapes the restriction. Conditions (H1)--(H3) are sufficient, not necessary, and the quantities they constrain are the ones the particular Lyapunov functional \eqref{eqn:energy} happens to control; a different functional, or a sharper treatment of the same one, could well admit a far larger set of initial data and parameters for the same dynamics. Two further gaps separate what is analyzed here from what is run in practice, which are the regularization of the best-position selections, and the passage from the discrete iteration \eqref{eqn:pso_all} to continuous time; quantifying either is beyond the present analysis. 

\section{Conclusion} \label{sec7}
We characterized the conditions on model parameters and initial conditions such that existing convergence guarantees still apply for the classical first-order CBO model (Theorem \ref{thm:pso_in_cbo}) and the second-order CBO model with memory (Theorem \ref{thm:cbo_mem_cv}) under the PSO coupling (\ref{eqn:coupling}). In particular, for the more intricate second-order model, we demonstrated that there exists a non-empty admissible set of parameters and initial laws under the PSO parameter coupling wherein all constraints are satisfied (Theorem \ref{thm:existence_family}). The qualification matters, as the initial laws our construction produces concentrate around the global minimizer, so the statement is one of local, not global, convergence. Beyond this, the parameter tradeoffs identified through our analysis are primarily qualitative, and the rigid coupling between drift and diffusion in the PSO reduction imposes significant constraints on the admissible parameter space. 
\newline
\newline
The variance-based convergence analysis, inherited from \cite{Huang_2023}, operates through a two-fold procedure: establishing consensus via variance decay ($V(\rho_t) \rightarrow 0$) and subsequently applying Laplace's principle to ensure the consensus point approximates the global minimizer. For the memory variant, this decay is stated in terms of the functional $\mathcal{H}(t)$ from \eqref{eqn:energy} rather than the variance directly. Writing $f_t := \mathrm{Law}(\overline{X}_t,\overline{Y}_t,\overline{V}_t)$ for the joint mean-field law of the model with memory, as in \eqref{eqn:cbo_mem_mf}, and $V(f_t) := \tfrac{1}{2}\mathbb{E}|\overline{X}_t - \mathbb{E}[\overline{X}_t]|^2$ for the variance of its position marginal, the leading term of $\mathcal{H}(t)$ in \eqref{eqn:energy} is precisely $(\gamma/2m)^2\cdot 2V(f_t)$, so controlling $\mathcal{H}(t)$ controls $V(f_t)$ directly.
\newline 
\newline 
These cross terms play a crucial role, since the noise in the mean-field dynamics \eqref{eqn:cbo_mem_mf} enters only through $\overline{V}_t$ (the updates for $\overline{X}_t$ and $\overline{Y}_t$ are noise-free), the position variance $V(f_t)$ may not decay monotonically on its own, and mixing position, personal best, and velocity is exactly the hypocoercivity approach used to restore a Lyapunov structure. The conditions (H1)--(H3), i.e.\ $\mu_1,\mu_2,\mu_3$ thus leads to the exponential decay of $\mathbb{E}[\mathcal{H}(t)]$ (Theorem \ref{thm:cbo_mem_cv}) be read as exponential decay of the mean-field variance $V(f_t)$. It also means that when (H1)--(H3) fail, what fails is the closure of this particular differential inequality: The conditions are sufficient only, so their failure implies neither loss of coercivity nor absence of convergence; it means that this route to a decay rate is unavailable and that some other argument would be needed. As evidenced by our numerical experiments, this framework necessitates a high degree of ``well-preparedness'' of the initial distribution $f_0$, and the requirement of an extremely small initial energy illustrates that the current theory provides a result of local convergence rather than a proof of global convergence. 
\newline
\newline
This offers a theoretical motivation for the heuristic ``Restart'' strategies frequently employed in engineering applications of PSO, though not a justification of them. In practice, when a swarm fails to converge or stagnates, the particles are re-initialized with a tighter distribution around the current best estimate; the analysis above shows that the region of initial data for which our sufficient conditions certify decay is of exactly this concentrated form. The analogy has clear limits: our conditions concentrate the initial law around the \emph{global minimizer} $x^*$ (Theorem \ref{thm:existence_family}), whereas a restart concentrates around the \emph{current best iterate}, which need not be near $x^*$, and the guarantee is in any case for the regularized mean-field model rather than the algorithm as implemented. Furthermore, our analysis suggests that the independence of drift and diffusion coefficients in the original CBO model is fundamental to its theoretical tractability. In PSO, the coupling $\sigma_i\propto c_i$ means that the same mechanism penalized in a variance-based Lyapunov analysis (diffusion growing with $c_i$) is also the mechanism that facilitates global exploration in practice.
\newline
\newline
Moving forward, an important open question is whether a Wasserstein-decay ($\mathcal{W}_2$) approach can be extended to the second-order CBO with memory variant. A $\mathcal{W}_2$ framework may offer more robust bounds that better capture the interplay between momentum, memory, and stochastic exploration, potentially providing a pathway toward a global convergence guarantee for the original particle swarm optimization algorithm. Existing convergence guarantees for the first-order CBO model in this framework were shown in \cite{Fornasier_2024}.

\subsubsection*{Acknowledgments}
FH, DK and RT were supported by NSF CAREER Award 2340762 and by Caltech start-up funds. RT acknowledges the Caltech Student-Faculty Programs, the Hugh F. and Audy Lou Colvin family for financial support during the summer of 2024, and the NSF GRFP. The authors are also grateful to Lorenzo Pareschi for helpful discussions and for suggesting the initial question.  

\bibliographystyle{abbrv}
\bibliography{references}

\end{document}